\documentclass{amsart}
\usepackage{amssymb,amsmath,amsthm}
\usepackage[foot]{amsaddr}
\title{Univalent polymorphism}
\date{\today}
\author{Benno van den Berg}
\address{Institute for Logic, Language and Computation (ILLC), University of Amsterdam, P.O. Box 94242, 1090 GE Amsterdam, the Netherlands. E-mail: B.vandenBerg3@uva.nl.}
\usepackage{ast2}
\usepackage{tikz}
\usetikzlibrary{cd}
\usepackage{proof}
\usepackage[margin=1.7in]{geometry}

\newcommand{\pullback}[2]{{}_{#1}\kern-\scriptspace{\times}_{#2}}
\begin{document}

\begin{abstract}
We show that Martin Hyland's effective topos can be exhibited as the homotopy category of a path category $\mathbb{EFF}$. Path categories are categories of fibrant objects in the sense of Brown satisfying two additional properties and as such provide a context in which one can interpret many notions from homotopy theory and Homotopy Type Theory. Within the path category $\mathbb{EFF}$ one can identify a class of discrete fibrations which is closed under push forward along arbitrary fibrations (in other words, this class is polymorphic or closed under impredicative quantification) and satisfies propositional resizing. This class does not have a univalent representation, but if one restricts to those discrete fibrations whose fibres are propositions in the sense of Homotopy Type Theory, then it does. This means that, modulo the usual coherence problems, it can be seen as a model of the Calculus of Constructions with a univalent type of propositions. We will also build a more complicated path category in which the class of discrete fibrations whose fibres are sets in the sense of Homotopy Type Theory has a univalent representation, which means that this will be a model of the Calculus of Constructions with a univalent type of sets.
\end{abstract}

\maketitle

\section{Introduction}

There has been a continuing interest in obtaining models of impredicative type theories, like the Calculus of Constructions (CoC). One reason is the use of CoC as the basis of proof assistants such as Coq and Lean; on a more theoretical level, these models shed light on issues surrounding impredicativity. Recently, the quest for models has received a new impetus due to the influence of Homotopy Type Theory (HoTT). Some natural questions which have arisen are whether it is possible to have universes which are both impredicative and univalent, how the type Prop of propositions in CoC relates to the notion of proposition in HoTT and whether one can obtain models of Voevodsky's resizing axioms.

So what is needed to model CoC? Roughly, what one needs is a category equipped with two classes of maps, one (``the small fibrations'') contained in the other (``the fibrations''), and where both classes are pullback stable, closed under composition and contain all isomorphisms. One needs a particular small fibration (a ``representation'') such that any other small fibration can be obtained as a pullback of that one; in addition, one needs to be able to push forward fibrations along other fibrations (where push forward is the right adjoint to pullback), and, crucially, one needs that small fibrations are closed under being pushed forward along arbitrary fibrations (this is what is meant by impredicativity in this context).

There is an old idea (see \cite{hylandetal90}) that models of CoC can be obtained by looking at Hyland's effective topos $\Eff$ \cite{hyland82}: the idea is that in this case every map is fibration and the small fibrations are the discrete maps, with a map in the effective topos being discrete if all its fibres are subcountable. It turns out that this does not quite work, due to a certain ambiguity in the notion of a discrete map, as explained in \cite{hylandetal90}. There is a global notion of a discrete map, where one says that a map $f: Y \to X$ in $\Eff$ is discrete if $f$ can be covered by a subobject of the natural numbers object $\pi_X: \mathbb{N} \times X \to X$ in $\Eff/X$ (this is how discrete maps are defined in \cite[Definition 3.4.2(i)]{vanoosten08}). In this case the discrete maps have a representation, but are not stable under push forward along arbitrary maps (see \cite[Proposition 7.5]{hylandetal90}). Alternatively, there is the local notion of a discrete map, where one says that a map $f: Y \to X$ is discrete if there is an epi $q: X' \to X$ such that $q^*f$ is covered by a subobject of the natural numbers object $\pi_{X'}: \mathbb{N} \times X' \to X'$ in $\Eff/X'$ (this is equivalent to saying that, internally, $f$ has the right lifting property with respect to $\Omega \to 1$ or $\nabla(2) \to 1$). In this case the discrete maps are stable under push forward along arbitrary maps, but it is still an open question whether the discrete maps defined in this local manner have a representation in $\Eff$. 

The standard way of solving this problem is to restrict to the $\lnot\lnot$-separated objects in the effective topos (the assemblies; see also \cite{carbonietal88,hyland88}). Within these assemblies the distinction between the two notions of a discrete map disappears (see \cite[Proposition 6.7]{hylandetal90}) and the discrete maps have a representation and are closed under push forward along arbitrary maps. This means that within the assemblies one can obtain a model of CoC, and to this day this remains the ``standard'' model for CoC.

This paper suggests an alternative solution. The first step in this direction is the observation that the the effective topos arises as the homotopy category of a certain path category, which we will denote $\mathbb{EFF}$. 

The notion of a path category, short for a category with path objects, was first introduced in \cite{bergmoerdijk18} and provides an abstract categorical setting for doing homotopy theory, in a manner similar to Quillen's model categories. Path categories are categories equipped with two classes of maps, fibrations and equivalences, satisfying a number of axioms which strengthen Ken Brown's axioms for a category of fibrant objects \cite{brown73} (in that sense the notion of a path category is similar to Joyal's notion of a tribe \cite{joyal17}, with the notion of a tribe being the stronger of the two). From a type-theoretic perspective, path categories provide models of ``propositional identity types'' (models of Martin-L\"of's identity types but with the computation rule formulated as a propositional equality; see \cite{vandenberg18}). The key notions from Homotopy Type Theory such as transport, univalence and hlevel, make sense in any path category, and therefore these notions make sense in $\mathbb{EFF}$ as well. 

It turns out that in $\mathbb{EFF}$ one can define a notion of discrete fibration which is stable under push forward along arbitrary fibrations. It is still unclear if the class of discrete fibrations has a representation in $\mathbb{EFF}$; however, if we restrict the hlevel of its fibres to propositions in the sense of HoTT, the discrete fibrations will have a representation and will still be be closed under push forward along arbitrary fibrations. In fact, its representation is essentially the subobject classifier in the effective topos and this representation will be univalent in $\mathbb{EFF}$. That means that with the class of propositional discrete fibrations as the small fibrations, $\mathbb{EFF}$ is a model of CoC with a univalent Prop. Given that the homotopy category of $\mathbb{EFF}$ is the effective topos, this may not be so surprising. However, the author of this paper thinks it is quite satisfying to see the notion of proposition from CoC, HoTT and the effective topos match up nicely like this. In addition, it is interesting to note that $\mathbb{EFF}$ is a model of propositional resizing. (This should be compared with the negative result due to Taichi Uemura \cite{uemura18}, showing that it fails in cubical assemblies.)

One way in which the model in $\mathbb{EFF}$ is somewhat poor is that its universe only contains propositions (in the sense of HoTT) and therefore excludes many data types like $\mathbb{N}$ or $\mathbb{N} \to \mathbb{N}$ (in that it differs from the standard model in assemblies). The difficulty is that these objects are not propositions, but actually live one hlevel higher: that is, they are sets in the sense of HoTT. For that reason I will also construct a more complicated path category  $\mathbb{EFF}_1$ in which the class of fibrations of discrete sets is an impredicative class of small fibrations with a univalent representation. This univalent universe will also contain such objects as $\mathbb{N}$ or $\mathbb{N} \to \mathbb{N}$.

There is a clear sense in which the construction of $\mathbb{EFF}_1$ is the natural next step after the construction of $\mathbb{EFF}$. The natural next step after that would be the construction of a path category $\mathbb{EFF}_2$ in which the class of discrete fibrations of groupoids would be closed under push forward along arbitrary fibrations and would have a univalent representation. In the limit one would expect to be able to construct a path category $\mathbb{EFF}_\infty$ in which the class of discrete fibrations (without any restriction on its hlevel) would be closed under push forward along arbitrary fibrations and have a univalent representation. This would provide an alternative way of salvaging the original idea of obtaining models of CoC using discrete fibrations.

At this point I should make two important disclaimers. Firstly, when I say that $\mathbb{EFF}$ and $\mathbb{EFF}_1$ are models of CoC, I am ignoring the usual coherence problems related to substitution. I suspect it would be relatively straightforward to solve this by some careful rephrasing in terms of categories with families \cite{dybjer96,hofmann97b}, but I am leaving this to future work. Secondly, and this may be the more essential point, $\mathbb{EFF}$ and $\mathbb{EFF}_1$ will be models of a version of CoC in which many definitional equalities have been replaced by propositional ones: this applies not just to the computation rule for the identity type, but to many computation rules. In fact, I am imagining a version of CoC in which the notion of definitional equality has been completely eliminated in favour of the notion of propositional equality. I am currently working on writing down a precise syntax for such a system, but this is still work in progress.  I suspect that a reader familiar with the categorical semantics of type theory can guess what it would look like on the basis of the categorical definitions in Section 2 below.

Throughout this paper we will rely on categorical language: familiarity with the language of type theory will not be assumed.

Finally, unless explicitly noted otherwise, the results obtained here are all constructive. However, since we will be interested in building models of impredicative type theories, our metatheory will be impredicative as well, which means that our metatheory will be something like intuitionistic Zermelo-Fraenkel set theory ({\bf IZF}) or the internal logic of an elementary topos with a natural numbers object (both can be seen as subsystems of {\bf ZF}). Results which rely on the axiom of choice will be marked with the label {\bf (ZFC)}.

\section{Path categories}

We will start this paper by discussing the notion of a path category, as introduced in \cite{bergmoerdijk18}, and by showing how many of the concepts in Homotopy Type Theory make sense in the context of a path category. The reader who wishes to see more details should consult Section 2 of \cite{bergmoerdijk18}.

\subsection{Definition} Let me first give the definition.

\begin{defi}{pathcat}
A \emph{path category} (short for a category with path objects) is a category \ct{C} equipped with two classes of maps, called \emph{fibrations} and \emph{equivalences}, respectively. A fibration which is also an equivalence will be called \emph{trivial}. If $X \to PX \to X \times X$ is a factorisation of the diagonal on $X$ as an equivalence followed by a fibration, then $PX$ is a \emph{path object} for $X$. In this case, we will often denote the equivalence $X \to PX$ by $r$ and the fibration $PX \to X \times X$ by $(s, t)$. In a path category the following axioms should be satisfied:
\begin{enumerate}
\item Isomorphisms are fibrations and fibrations are closed under composition.
\item \ct{C} has a terminal object 1 and $X \to 1$ is always a fibration.
\item The pullback of a fibration along any other map exists and is again a fibration.
\item Isomorphisms are equivalences and equivalences satisfy 6-for-2: that is, if $f, g, h$ are maps for which the composition $hgf$ is defined and for which $hg$ and $gf$ are equivalences, then $f,g,h$ and $hgf$ are equivalences as well.
\item Every object $X$ has at least one path object.
\item Trivial fibrations are stable under pullback along arbitrary maps.
\item Trivial fibrations have sections.
\end{enumerate}
\end{defi}

\begin{rema}{onequivalences} In \cite{vandenberg18,bergmoerdijk18} the equivalences were called weak equivalences. Since every object in a path category is both fibrant and cofibrant (to use the language of model categories) and we have \reftheo{eqishomeq} below, there does not really seem to be much need for the adjective ``weak'', so we have decided to drop it. Indeed, the HoTT book \cite{univalent13} simply calls this class of maps equivalences, without the adjective.
\end{rema}

A basic fact about path categories is that they are stable under slicing.

\begin{defi}{Slicing for path categories}
Let \ct{C} be a path category and $X$ be an object in \ct{C}. Write $\ct{C}(X)$ for the full subcategory of $\ct{C}/X$ whose objects are fibrations. With the equivalences and fibrations as in \ct{C}, this is again a path category.
\end{defi}

\begin{prop}{pullingback}
If $f: Y \to X$ is a morphism in a path category, then the functor \[ f^*: \ct{C}(X) \to \ct{C}(Y), \] obtained by pulling back fibrations along $f$, preserves all the structure of a path category: that is, it preserves fibrations, equivalences, the terminal object and pullbacks of fibrations along arbitrary maps.
\end{prop}
\begin{proof} See \cite[Proposition 2.6]{bergmoerdijk18}.
\end{proof}

\subsection{Homotopy in path categories} On parallel morphisms in a path category one can define a homotopy relation, as follows.

\begin{defi}{homotopy}
If $f, g: Y \to X$ are two parallel maps in a path category, then we say that $f$ and $g$ are \emph{homotopic} if there is a map $H: Y \to PX$ to a path object on $X$ making
\begin{center}
\begin{tikzcd}
     & PX \arrow[d, "{(s,t)}"] \\
Y \arrow[r, "{(f,g)}"'] \arrow[ur, dotted, "H"] & X \times X
\end{tikzcd}
\end{center}
commute. (One can show that this is independent of the choice of path object $PX$.) In this case we write $f \simeq g$, or $H: f \simeq g$ if we wish to stress the \emph{homotopy} $H$. More generally, suppose $p: X \to I$ is a fibration and $P_I(X)$ is the path object of $X$ in the slice category $\mathcal{C}(I)$. If $f, g: Y \to X$ are two parallel maps with $pf = pg$, then we say that $f$ and $g$ are \emph{fibrewise homotopic} if there is a map $H: Y \to P_I(X)$ making
\begin{center}
\begin{tikzcd}
 & P_IX \arrow[d,"{(s, t)}"] \\
Y \ar[r, "{(f, g)}"'] \ar[ur, dotted, "H"] & X \times_I X
\end{tikzcd}
\end{center}
commute. (This is again independent of the choice of path object $P_I(X)$.) In this case we write $f \simeq_I g$, or $H: f \simeq_I g$ if we wish to stress the \emph{fibrewise homotopy} $H$.
\end{defi}

\begin{theo}{homcongr}
The homotopy relation $\simeq$ is a congruence on \ct{C}.
\end{theo}
\begin{proof}
This is Theorem 2.14 in \cite{bergmoerdijk18}.
\end{proof}

The quotient is the \emph{homotopy category} of \ct{C}. A map which becomes an isomorphism in the homotopy category is called a \emph{homotopy equivalence}.

\begin{theo}{eqishomeq}
The equivalences and homotopy equivalences coincide in a path category.
\end{theo}
\begin{proof}
This is Theorem 2.16 in \cite{bergmoerdijk18}.
\end{proof}

Using the notion of homotopy we can strengthen \reftheo{homcongr}.

\begin{prop}{everyobjgroupoiduptohom}
 Every object in a path category carries a groupoid structure up to homotopy. More precisely, if $A$ is an object in a path category and $PA$ is a path object for $A$ with equivalence $r: A \to PA$ and fibration $(s,t): PA \to A \times A$, while $PA \times_A PA$ is the pullback
\begin{center}
\begin{tikzcd}
 PA \times_A PA \arrow[d, "p_1"'] \arrow[r,"p_2"] & PA \arrow[d,"t"] \\
PA \ar[r, "s"'] & A,
\end{tikzcd}
\end{center}
then there are maps $\mu: PA \times_A PA \to PA$ and $\sigma: PA \to PA$ with $(s,t)\mu = (sp_2,tp_1)$ and $(s,t)\sigma = (t,s)$ such that:
\begin{enumerate}
\item $\mu(p_1,\mu(p_2,p_3)) \simeq_{A \times A} \mu(\mu(p_1,p_2),p_3): PA \times_A PA \times_A PA \to PA$.
\item $\mu(1, rs) \simeq_{A \times A} 1: PA \to PA$.
\item $\mu(rt,1) \simeq_{A \times A} 1: PA \to PA$.
\item $\mu(1,\sigma) \simeq_{A \times A} rt: PA \to PA$.
\item $\mu(\sigma, 1) \simeq_{A \times A} rs: PA \to PA$.
\end{enumerate}
\end{prop}
\begin{proof}
See \cite[Proposition A.6]{vandenberg18}.
\end{proof}

\begin{rema}{inftygroupoid}
 In fact, the arguments in \cite{vandenberggarner11,bourke16} can be modified to show that every object in a path category carries the structure of a Grothendieck $\infty$-groupoid.
\end{rema}

\subsection{Transport and univalence} A crucial concept in homotopy type theory is transport.

\begin{defi}{transport}
Let $f: Y \to X$ be a fibration, and $Y \times_X PX$ be the following pullback
\begin{center}
\begin{tikzcd}
Y \times_X PX \ar[d, "p_1"'] \ar[r, "p_2"] & PX \ar[d, "s"] \\
Y \ar[r, "f"'] & X
\end{tikzcd}
\end{center}
A \emph{transport structure} $\Gamma$ on $f$ is a map 
\[ \Gamma: Y \times_X PX \to Y \] such that $f \Gamma = tp_2$ and $\Gamma(1_Y, rf) \simeq_X 1_Y$.
\end{defi}

The intuition behind transport is the following: given a fibration $f: Y \to X$, an element $y_0 \in Y$ and a path $\alpha: x_0 \to x_1$ in $X$ with $f(y_0) = x_0$, transport finds an element $y_1$ in $Y$ with $f(y_1) = x_1$. In addition, we require that if $\alpha$ would the reflexivity path at $x_0$, then $y_1$ should be connected to $y_0$ by a path lying in the fibre over $x_0$.

\begin{theo}{existenceanduniquenessoftransport}
Every fibration $f: Y \to X$ carries a transport structure and transport structures are unique up to fibrewise homotopy over $X$.
\end{theo}
\begin{proof}
This is Theorem 2.26 in \cite{bergmoerdijk18}.
\end{proof}

In the sequel we will frequently rely on the following properties of transport.

\begin{prop}{basicpropoftransport}
 Let $f: Y \to X$ be a fibration with transport structure \[ \Gamma: Y \times_X PX \to Y. \] Then:
\begin{enumerate}
 \item $\Gamma$ preserves fibrewise homotopy. More precisely, the maps \[ \Gamma(sp_1,p_2),\Gamma(tp_1,p_2):P_X(Y) \times_X PX \to Y \] are fibrewise homotopic.
 \item If you have two paths in $X$ which have the same end points and which are homotopic relative those end points, then the transports of $y \in Y$ along both paths are fibrewise homotopic. More precisely, the maps
\[ \Gamma(p_1,sp_2), \Gamma(p_1,tp_2):Y \times_X P_{X \times X}(PX) \to Y \]
are fibrewise homotopic.
\item The transport of an element in $Y$ along the composition of two paths and the result of first transporting it along one path and then along the other are fibrewise homotopic. More precisely, the maps
\[ \Gamma(p_1,\mu(p_3,p_2)), \Gamma(\Gamma(p_1,p_2),p_3): Y \times_X PX \times_X PX \to Y \]
are fibrewise homotopic.
\end{enumerate}
\end{prop}
\begin{proof}
See \cite[Lemma A.10-12]{vandenberg18}.
\end{proof}

\begin{prop}{transportisanequivalence}
Suppose $p: Y \to X$ is a fibration and $f, g: Z \to X$ are two parallel maps which are homotopic via a homotopy $H: Z \to PX$. The morphism $f^* p \to g^* p$ over $Z$ induced by $H$ and the transport structure on $p$ is an equivalence. 
\end{prop}
\begin{proof} To be clear, if $H: Z \to PX$ is the homotopy and $\Gamma: Y \times_X PX \to Y$ is a transport structure, the induced map $f^*p \to g^* p$ is:
\begin{center}
 \begin{tikzcd}
  Z \pullback{f}{p} Y \ar[rr, "1 \times_X H\pi_Z"] & & Z \times_X Y \times_X PX \ar[rr, "Z \times_X \Gamma"] & & Z \pullback{g}{p} Y.
 \end{tikzcd}
\end{center}
Since every path in $PX$ has an inverse up to homotopy, there is also an operation $g^* p \to f^*p$ induced by the inverse of the homotopy $H$. Using the properties of transport from the previous proposition one can show that this operation is homotopy inverse to the one induced by $H$.
\end{proof}

A fibration is univalent if we can also go in the other direction, as follows.

\begin{defi}{univalence} The fibration $p: Y \to X$ is \emph{univalent} if the following holds: whenever $f, g: Z \to X$ are two morphisms and $w: f^* p \to g^* p$ is an equivalence in the slice over $Z$, then there is a homotopy $H: Z \to PX$ between $f$ and $g$ such that $w$ is, up to fibrewise homotopy over $Z$, the equivalence induced by $H$, as in the previous proposition.
\end{defi}

\begin{rema}{onunivalence} In other words, if $p: Y \to X$ is a fibration, then any path $\pi: x \to x'$ in $X$ induces an equivalence between the fibres $Y_x \simeq Y_{x'}$; a fibration is univalent if, up to fibrewise homotopy, every equivalence between the fibres is obtained in this way. This may look a bit different from  the way univalence is usually stated, but it is equivalent to the standard formulation, as discussed on the HoTT mailing list. See also \cite[Theorem 3.5]{ortonpitts18} for more details.
\end{rema}

\subsection{Function spaces} No type theory is complete without function types. We will work with function types in the following form, that is, they will be exponentials in the homotopy category.

\begin{defi}{hexpo}
If $X$ and $Y$ are objects in a path category, then we will call an object $X^Y$ the \emph{homotopy exponential} of $X$ and $Y$ if it comes equipped with a map ${\rm ev}: X^Y \times Y \to X$ such that for any map $h: A \times Y \to X$ there is a map $H: A \to X^Y$ such that ${\rm ev}(H \times 1_Y) \simeq h$, with $H$ being unique up to homotopy with this property.
\end{defi}

\begin{rema}{ondefihomexp} Perhaps two remarks are in order about this definition of a homotopy exponential from a type-theoretic perspective. First of all, we only demand that ${\rm ev}(H \times 1_Y)$ and $h$ are homotopic, rather than equal. This basically states $\beta$-equality in a propositional, rather than a definitional form: this reflects the fact that we are thinking of a version of type theory in which all the computational rules have been stated propositionally. Secondly, the fact that $H$ is unique up to homotopy amounts to the validity of function extensionality. (See also \cite[Remark 5.8]{bergmoerdijk18}.) Similar remarks apply to the following definition of homotopy $\Pi$-types.
\end{rema}

\begin{defi}{hPitypes} If $g: Z \to Y$ and $f: Y \to X$ are fibration, then the \emph{homotopy $\Pi$-type} of $f$ along $g$ is a fibration $\Pi_f(g): E \to X$ equipped with a map $\varepsilon: f^*\Pi_f(g) \to g$ over $Y$ such that for any $h: A \to X$ and $m: f^*h \to g$ over $Y$ there exists a map $M: h \to \Pi_f(g)$ such that $\varepsilon \circ f^*M \simeq_Y m$, with $M$ being unique up to fibrewise homotopy over $X$ with this property. We will say that a path category \ct{C} has homotopy $\Pi$-types if all such homotopy $\Pi$-types exist for fibrations $f$ and $g$.
\end{defi}

One can phrase this in terms of right adjoints but one has to be a bit careful. If \ct{C} is a path category with homotopy $\Pi$-types and $f: Y \to X$ is a fibration in \ct{C}, then the pullback functor $f^*: \ct{C}(X) \to \ct{C}(Y)$ has a right adjoint if we pass to the homotopy category. That is, there is an operation $\Pi_f: \ct{C}(Y) \to \ct{C}(X)$ which goes in the other direction, but it is only a functor up to fibrewise homotopy over $X$. 

Also, we cannot really say that the operation $\Pi_f: \ct{C}(Y) \to \ct{C}(X)$ preserves fibrations. We do have the following statement, however, which is a good substitute.

\begin{prop}{goodchoiceofpie}
 Let \ct{C} be a path category with homotopy $\Pi$-types. Assume $f: Y \to X$ is a fibration in $\ct{C}$, while $p: B \to A$ is a fibration in $\ct{C}(Y)$ and $\Pi_f(A)$ is a homotopy $\Pi$-type in $\ct{C}(X)$. Then there exists a homotopy $\Pi$-type $\Pi_f(B)$ in $\ct{C}(X)$ together with a fibration $\Pi_f(p): \Pi_f(B) \to \Pi_f(A)$ in $\ct{C}(X)$ such that:
\begin{enumerate}
 \item The diagram
  \begin{center}
 \begin{tikzcd}
  f^*\Pi_f(B) \ar[d,"f^*\Pi_f(p)"'] \ar[r, "\varepsilon"] & B \ar[d, "p"] \\
  f^*\Pi_f(A) \ar[r, "\varepsilon"'] & A \\
 \end{tikzcd}
\end{center}
  in $\ct{C}(Y)$ commutes.
 \item For each object $C$ in $\ct{C}(X)$ the diagram
 \begin{center}
 \begin{tikzcd}
  {\rm Ho}(\ct{C}(X))(C, \Pi_f(B)) \ar[d] \ar[r] & {\rm Ho}(\ct{C}(Y))(f^*C, B) \ar[d] \\
  {\rm Ho}(\ct{C}(X))(C, \Pi_f(A)) \ar[r] & {\rm Ho}(\ct{C}(Y))(f^*C, A) \\
 \end{tikzcd}
\end{center}
  is a pullback in $\Sets$.
\end{enumerate}
\end{prop}
\begin{proof} (Compare \cite[Proposition 5.5]{bergmoerdijk18}.) We construct $\Pi_f(p)$ as follows. First, we take the pullback
\begin{center}
 \begin{tikzcd}
  P \ar[d, "q"'] \ar[r] & B \ar[d, "p"] \\
  f^*\Pi_f(A) \ar[r, "\varepsilon"'] & A \\
 \end{tikzcd}
\end{center}
and then we define $\Pi_f(p)$ to be $\Pi_{\pi}(q)$, where $\pi: f^*\Pi_f(A) \to \Pi_f(A)$ is the projection.
\end{proof}

\subsection{Homotopy $n$-types} One can also makes sense of Voevodsky's homotopy levels in the context of path categories.

\begin{defi}{hlevels}
The \emph{fibrations of $n$-types} ($n \geq -2$) are defined inductively as follows:
\begin{itemize}
\item A fibration $f: Y \to X$ is a fibration of $(-2)$-types if $f$ is trivial. 
\item A fibration $f: Y \to X$ is a fibration of $(n+1)$-types if $P_X(Y) \to Y \times_X Y$ is a fibration of $n$-types for some path object $P_X(Y)$ for $f$ in $\ct{C}(X)$.
\end{itemize}
An object $X$ in a path category is called an \emph{$n$-type} if the map $X \to 1$ is a fibration of $n$-types.
\end{defi}

\begin{defi}{namesforhlevels} For small $n$ we will often use the following terminology for $n$-types, as is customary in the context of HoTT.
\begin{center}
\begin{tabular}{|c|c|} \hline
(-2)-types & contractible types \\
(-1)-types  & propositions \\
0-types & sets \\
1-types & groupoids \\ \hline
\end{tabular}
\end{center}
\end{defi}

In the definition of a fibration of $(n+1)$-types the dependence on the choice of path object $P_X(Y)$ is only apparent. The reason for this is that path objects are unique up to equivalence (see \cite[Corollary 2.39]{bergmoerdijk18}) and we have the following proposition.

\begin{prop}{hlevelsstableunderpbk} Let $f$ be a fibration of $n$-types.
\begin{enumerate}
\item If $g$ is a fibration which can be obtained by pulling back $f$ along some map, then $g$ is also a fibration of $n$-types.
\item If $w$ is an equivalence and $g$ is a fibration such that $f \simeq gw$ or $fw \simeq g$, then $g$ is also a fibration of $n$-types.
\item If $w$ is an equivalence and $g$ is a fibration such that $f \simeq wg$ or $wf \simeq g$, then $g$ is also a fibration of $n$-types.
\end{enumerate}
\end{prop}
\begin{proof}
(1) is easily proved by induction on $n \geq -2$ using \refprop{pullingback}.

(2) It suffices to prove that if $g = fw$ and $w$ is an equivalence, then $g$ is a fibration of $n$-types. To see this, note, first of all, that if $g \simeq fw$, then there is a map $w'$ homotopic to $w$ such that $g = fw'$ (see \cite[Proposition 2.31]{bergmoerdijk18}); such a map $w'$ is an equivalence as well, because it is homotopic to $w$. Furthermore, if $gw \simeq f$, then $g \simeq fw^{-1}$ for a homotopy inverse $w^{-1}$ of $w$; and such a homotopy inverse is an equivalence as well.

So we prove that if $g = fw$ and $w: Z \to Y$ is an equivalence and $f: Y \to X$ and $g: Z \to X$ are fibrations, with $f$ being a fibration of $n$-types, then $g$ is a fibration of $n$-types as well. This statement is proved by induction on $n$, with the case $n = 0$ following from 6-for-2 for equivalences.

So assume the statement is true for $n$ and $f$ is a fibration of $(n+1)$-types. We have a commuting diagram of the form
\begin{center}
 \begin{tikzcd}
  Z \ar[r] \ar[d, "w"'] & P_X(Z) \ar[r] \ar[d, dotted] & Z \times_X Z \ar[d, "w \times_X w"] \\
Y \ar[r] & P_X(Y) \ar[r] & Y \times_X Y.
 \end{tikzcd}
\end{center}
In this diagram we have a dotted arrow down the middle which make the diagram on the left commute up to homotopy, while the diagram of the right commutes on the nose (see \cite[Theorem 2.38]{bergmoerdijk18}). The dotted arrow is an equivalence because all the other arrows in the left hand square are. In addition, because equivalences are stable under pullback along fibrations (see \cite[Proposition 2.7]{bergmoerdijk18}), we have that $w \times_X w$ and the inscribed map $P_X(Z) \to (w \times_X w)^* P_X(Y)$ are equivalences as well. By assumption, the map $P_X(Y) \to Y \times_X Y$ is a fibration of $n$-types, so by (1) the same is true for its pullback along $w \times_X w$, and hence the same is true for $P_X(Z) \to Z \times_X Z$ by induction hypothesis. Hence $g$ is a fibration of $(n+1)$-types.

(3) Again, it suffices to prove only one of the two statements and we will show that $f \simeq wg$ implies that $g$ is a fibration of $n$-types. So assume $f \simeq wg$ and construct the following pullback
\begin{center}
\begin{tikzcd}
 D \ar[r, "v"] \ar[d, "h"'] & B \ar[d, "f"] \\
C \ar[r, "w"'] & A.
\end{tikzcd}
\end{center}
We have $wgv \simeq fv = wh$ and since $w$ is an equivalence, it follows that $h \simeq gv$. Since $v$ is an equivalence by \cite[Proposition 2.7]{bergmoerdijk18} and $h$ is a fibration of $n$-types by (1), this implies that $g$ is a fibration of $n$-types by (2).
\end{proof}

\begin{theo}{hlevelsandPi}
If \ct{C} is a path category with homotopy $\Pi$-types, then $\Pi_f:\ct{C}(Y) \to \ct{C}(X)$ preserves fibrations of $n$-types for any fibration $f: Y \to X$.
\end{theo}
\begin{proof}
This is Theorem 7.1.9 in the HoTT book \cite{univalent13}. The proof is by induction on $n \geq -2$. 

We have seen that $\Pi_f$ is not really a functor, but it is one up to homotopy. For that reason it preserves equivalences and therefore the statement is true for $n = -2$.

For the induction step, suppose that $p: Z \to Y$ is a fibration on $(n+1)$-types in $\ct{C}(Y)$; so $P_Y(Z) \to Z \times_Y Z$ is a fibration of $n$-types. Since $\Pi_f(Z) \times_X \Pi_f(Z)$ is a homotopy $\Pi$-type of $Z \times_Y Z$ along $f$, we can compute $\Pi_f(P_Y(Z))$ as in the proof of \refprop{goodchoiceofpie}, which, by induction hypothesis and the previous proposition, ensures that 
\[ \Pi_f(P_Y(Z))) \to \Pi_f(Z) \times_X \Pi_f(Z) \]
will be a fibration of $n$-types. Since $\Pi_f$ preserves equivalences, $\Pi_f(P_Y(Z))$ is a path object for $\Pi_f(Z)$ in $\ct{C}(X)$ and therefore $\Pi_f(Z) \to X$ is a fibration of $(n+1)$-types, as desired.
\end{proof}

\begin{defi}{ntruncation}
The \emph{$n$-truncation} of a fibration $f: B \to A$ is a factorisation of $f$ as a map $g: B \to C$ followed by a fibration on $n$-types $h: C \to A$, in such a way that if $g': B \to C'$ and $h': C' \to A$ is another such factorisation, then there exists a map $d: C \to C'$ over $A$ with $dh \simeq_A g'$; if, moreover, $d': C \to C'$ would be another such map then $d \simeq_A d'$. 
\end{defi}

\subsection{Small fibrations} Let us suppose $\ct{S}$ is a subclass of the class of fibrations which is closed under composition, contains all isomorphisms and is stable under homotopy pullback. (Recall that a commutative square
\begin{center}
\begin{tikzcd}
D \ar[r] \ar[d, "g"'] & B \ar[d, "f"] \\
C \ar[r] & A 
\end{tikzcd}
\end{center}
in which both $f$ and $g$ are fibrations is a \emph{homotopy pullback} if the induced map $D \to C \times_A B$ is an equivalence. By saying that $\ct{S}$ is stable under homotopy pullback, we mean that whenever there is such a square and $f$ belongs to $\ct{S}$, then $g$ belongs to $\ct{S}$ as well.) Let us refer to the elements of $\ct{S}$ as the ``small fibrations''.

\begin{defi}{representation} Let $\ct{S}$ be a class of small fibrations, as above. A \emph{representation} for such a class $\ct{S}$ is a small fibration $\pi: E \to U$ such that any other small fibration $f: B \to A$ can be obtained as a homotopy pullback of that one via some map $A \to U$.
\end{defi}
 
We will be especially interested in classes of small fibrations with a univalent representation. Note that for a univalent representation $\pi: E \to U$ the classifying map $A \to U$ is unique up to homotopy.

For obtaining models of CoC, we need to assume in addition that the class of small fibrations is impredicative.

\begin{defi}{impredicative}
Let $\ct{S}$ be a class of small fibrations in a path category with homotopy $\Pi$-types. Then $\ct{S}$ is called \emph{impredicative} (or \emph{polymorphic}) if it is closed under $\Pi_f$ for \emph{any} fibration $f$.
\end{defi}

Finally, we can also formulate Voevodsky's resizing axiom.

\begin{defi}{resizing} A class of small fibrations $\ct{S}$ satisfies \emph{(propositional) resizing} if every propositional fibration belongs to $\ct{S}$.
\end{defi}

\section{The path category $\mathbb{EFF}$}

In this section we will see that Martin Hyland's effective topos \cite{hyland82} is the homotopy category of an interesting path category. This path category will be called $\mathbb{EFF}$ and is constructed as follows.

Objects of the category $\mathbb{EFF}$ consist of:
\begin{enumerate}
\item A set $A$.
\item A function $\alpha: A \to \mathbb{N}$ (sending an element $a \in A$ to its \emph{realizer}).
\item For each pair of elements $a,a' \in A$ a subset $\mathcal{A}(a,a')$ of $\mathbb{N}$.
\item A function which computes for a realizer of $a \in A$ an element in $1_a \in \mathcal{A}(a,a)$.
\item A function which given realizers for $a,a'$ and $\pi \in \mathcal{A}(a,a')$ computes an element $\pi^{-1} \in \mathcal{A}(a', a)$.
\item A function which given realizers for $a,a',a''$ and $\pi \in \mathcal{A}(a,a'), \pi' \in \mathcal{A}(a',a'')$ computes an element $\pi' \circ \pi \in \mathcal{A}(a,a'')$.
\end{enumerate}

Objects in this category will typically be denoted by $(A, \alpha, \mathcal{A})$, suppressing the data in  (4-6).

\begin{rema}{onobjinEFF1}
Just to be clear, condition (4) has to be read in the following way: there is a partial computable function, which given any element $a \in A$ computes an element in $\mathcal{A}(a,a)$ from the realizer of $a$.  Conditions (5) and (6) have to be read in the same way. In fact, from a certain point onwards we will simply state a condition like (4) as: an element in $\mathcal{A}(a,a)$ can be computed from $a \in A$, where it has to be understood that the partial computable function gets the realizer of $a$ as input.
\end{rema}

A morphism $f: (B, \beta, \mathcal{B}) \to (A, \alpha, \mathcal{A})$ in $\mathbb{EFF}$ consists of:
\begin{enumerate}
\item A function $f: B \to A$ such that the realizer of $f(b)$ can be computed from a realizer of $b$.
\item For each $b, b' \in B$ a function $f_{(b, b')}: \mathcal{B}(b,b') \to \mathcal{A}(fb, fb')$ such that $f_{(b,b')}(\pi)$ can be computed from realizers for $b,b'$ and $\pi$.
\end{enumerate}
If $f: (B, \beta, \mathcal{B}) \to (A, \alpha, \mathcal{A})$ is such a morphism, we will refer to G\"odel numbers for the partial computable functions whose existence is demanded in (1) and (2) as \emph{trackings} for the 0- and 1-cells, respectively. 

\begin{rema}{thinkingofEFF} The right way to think about the objects and morphisms in $\mathbb{EFF}$ is as locally codiscrete bigroupoids and (non-strict) homomorphisms between them, with elements in $A$ as the 0-cells and the elements in $\mathcal{A}(a,a')$ as the 1-cells. The definitions do not mention 2-cells, but one should really think of any two elements in $\mathcal{A}(a,a')$ as being connected by a unique 2-cell. So one could add certain (weak) groupoid laws and demand that the morphisms preserve them, but with this intuition in mind it becomes clear that such conditions would be superfluous.
\end{rema}

\begin{defi}{homotopyinEFF}
Two parallel maps $f, g: (B, \beta, \mathcal{B}) \to (A, \alpha, \mathcal{A})$ are \emph{homotopic} if there is a function computing for every realizer for $b \in B$ an element in $\mathcal{A}(fb,gb)$ (such a function is called a \emph{homotopy}). A G\"odel number for such a function will be called a \emph{coded homotopy}. A map $f: (B, \beta, \mathcal{B}) \to (A, \alpha, \mathcal{A})$ is a \emph{(homotopy) equivalence} if there is a morphism $g$ in the other direction (the homotopy inverse) such that both composites $fg$ and $gf$ are homotopic to the identity.
\end{defi}

Note that the homotopy relation does not take into account the action of morphisms on 1-cells.

\begin{defi}{fibrinEFF}
A map $f:  (B, \beta, \mathcal{B}) \to (A, \alpha, \mathcal{A})$ in $\mathbb{EFF}$ is a \emph{fibration} if:
\begin{enumerate}
\item for any $b, a$ and $\pi: f(b) \to a$ one can effectively find $b', \rho: b \to b'$ such that $f(b') = a$ and $f(\rho) = \pi$ (meaning: there are partial computable functions $\varphi$ and $\psi$ such that for any $b, a$ and $\pi: f(b) \to a$ there are $b', \rho: b \to b'$ with $f(b') = a$ and $f(\rho) = \pi$ and with $\varphi$ computing the realizer of $b'$ from $\pi$ and realizers for $a$ and $b$ and with $\psi$ computing $\rho$ from $\pi$ and realizers for $a$ and $b$).
\item for any $b, b' \in B$, $\rho \in \mathcal{B}(b,b')$ and $\pi \in \mathcal{A}(fb,fb')$ one can compute $\rho' \in \mathcal{B}(b,b')$ with $f(\rho') = \pi$.
\end{enumerate}
\end{defi}

\begin{rema}{secondcondition} The second (somewhat remarkable) condition here is easy to miss, but with a bigroupoidal intuition it is clear that it should be there: it states a lifting condition for the invisible 2-cells.
\end{rema}

\begin{theo}{EFFisapathcat}
With the equivalences and fibrations as defined above, $\mathbb{EFF}$ is a path category.
\end{theo}
\begin{proof} We verify the axioms. It is easy to see that axioms 1 and 2 for a path category are satisfied, and because the  homotopy relation is a congruence,  axiom 4 follows as well.

Let us first construct pullbacks of fibrations. If $f: (B, \beta, \mathcal{B}) \to (A, \alpha, \mathcal{A})$ is a fibration and $g: (C, \gamma, \mathcal{C}) \to (A, \alpha, \mathcal{A})$ is any map, then its pullback is:
\begin{eqnarray*}
D & = & \{ \, (c,b) \, : \, g(c) = f(b) \, \} \\
\delta(c,b) & = & (\gamma(c), \beta(b)) \\
\mathcal{D}((c,b),(c',b')) & = & \{ (\pi \in \mathcal{C}(c,c'), \rho \in \mathcal{B}(b,b')) \, : \, g(\pi) = f(\rho) \, \}
\end{eqnarray*}
Note that one has to use the second condition in the definition of a fibration to define the bigroupoid structure on $(D, \delta, \mathcal{D})$: for instance, if $d = (c,b) \in D$, then there are 1-cells $1_b \in \mathcal{B}(b,b)$ and $1_c \in \mathcal{C}(c,c)$. We have $g(1_c) \in \mathcal{A}(fb,fb)$, so by the second condition there is a 1-cell $\pi \in \mathcal{B}(b,b)$ with $f(\pi) = g(1_c)$ and hence $1_d := (1_c,\pi) \in \mathcal{D}(d,d)$. For the inverses and composition there are similar arguments. In addition, the obvious projection $(D, \delta, \mathcal{D}) \to (C, \gamma, \mathcal{C})$ is again a fibration (this takes care of axiom 3).

Let us now construct path objects (axiom 5). If $(A, \alpha, \mathcal{A})$ is any object, then its path object $(P, \pi, \mathcal{P})$ is:
\begin{eqnarray*}
P & = & \{ \, (a,a', \rho) \, : \, a,a' \in A, \rho \in \mathcal{A}(a,a') \, \} \\
\pi(a,a',\rho) & = & ( \alpha(a), \alpha(a'), \rho ) \\
\mathcal{P}((a,a',\rho),(b,b',\sigma)) & = & \{ \langle m, n \rangle \, : \, m \in \mathcal{A}(a,a'), n \in \mathcal{A}(b,b') \}
\end{eqnarray*}
Note that the notion of homotopy determined by these path objects is precisely the homotopy relation we defined earlier.

Finally, we make the following \emph{claim}: a fibration $f: (B, \beta, \mathcal{B}) \to (A, \alpha, \mathcal{A})$ is trivial precisely when it has a section $s$ such that $sf$ is homotopic to the identity. Then it is easy to see that trivial fibrations have sections and are stable under pullback (axioms 6 and 7). So suppose that $g: A \to B$ is a morphism such that $H: fg \simeq 1$. Then for every $a \in A$ there is a morphism $H_a: fga \to a$, and the first property of a fibration allows us to find an element $sa \in B$ and a morphism $K_a: ga \to sa$ such that $f(K_a) = H_a$. If $\pi: a \to a'$ is a 1-cell in $A$, then we can use the fact that $g$ is a morphism to find a 1-cell $sa \to sa'$. Using the second condition for a fibration we can choose $s(\pi): sa \to sa'$ such that $fs(\pi) = \pi$. This constructs a section $s$ of $f$, which is homotopic to $g$. This shows that the \emph{claim} about trivial fibrations is correct.
\end{proof}

\begin{prop}{hoisEff}  The homotopy category of $\mathbb{EFF}$ is equivalent to the effective topos.
\end{prop}
\begin{proof}
We will assume familiarity with the effective topos $\Eff$ (see, for instance, \cite{vanoosten08}). Consider the functor \[ P: \mathbb{EFF} \to \Eff \]
sending $(A, \alpha, \mathcal{A})$ to $(A, =_A)$ with
\[ a =_A a' = \{ <n,\pi,n'> \, : \, n = \alpha(a), n' = \alpha(a'), \pi \in \mathcal{A}(a,a') \} \]
and $f: (B, \beta, \mathcal{B}) \to (A, \alpha, \mathcal{A})$ to $F: A \times B \to {\rm Pow}(\mathbb{N})$ with
\[ F(b,a) = \{ <m, \pi, n> \, : \, m = \beta(b), n = \alpha(a), \pi \in \mathcal{A}(fb,a). \} \]
This functor $P$ is full and for two parallel arrows $f, g$ in $\mathbb{EFF}$ we have $Pf = Pg$ precisely when $f$ and $g$ are homotopic. So it remains to verify that $P$ is essentially surjective.

So let $(X, =_X)$ be an object in the effective topos. If $(A, \alpha, \mathcal{A})$ is the object in $\mathbb{EFF}$ with
\begin{eqnarray*}
A & = & \{ <x,n> \, : \, x \in X, n \in x =_X x \} \\
\alpha(<x,n>) & = & n \\
\mathcal{A}(<x,n>,<x',n'>) & = & x =_X x'
\end{eqnarray*}
then $P(A, \alpha, \mathcal{A}) \cong (X, =_X)$ in $\Eff$.
\end{proof}

\begin{rema}{onchoiceandEFF} In a metatheory without the axiom of choice, the effective topos can be defined in different ways. To prove that the functor $P$ from the previous proof is full, we need a version where functionality of a morphism $F: (B, =_B) \to (A, =_A)$ in $\Eff$ means that we have a function $f$ picking for every $b \in B$ and $n \in b=_B b$ an element $a = f(b, n) \in A$ such that from $n$ one can compute an element in $F(b,a)$.
\end{rema}

\begin{rema}{onrosolini} The contents from this section are very close to an earlier argument by Rosolini in \cite{rosolini16}, where he also constructs the effective topos as a homotopy quotient. The main difference is that here we obtain the effective topos as the homotopy category of a path category, which also involves finding a suitable class of fibrations; in addition, we think of the objects in the path category as degenerate bigroupoids, rather than as 2-groupoids.
\end{rema}

\begin{rema}{alsofromearlierpaper}
It also follows from \cite{bergmoerdijk18} that the effective topos arises as a homotopy category: since the effective topos is the exact completion of the category of partitioned assemblies (see \cite{robinsonrosolini90}), the effective topos is the homotopy category of ${\rm Ex}({\bf PAsm})$, where ${\rm Ex}$ is the homotopy exact completion from \cite{bergmoerdijk18} and ${\bf PAsm}$ is the category of partitioned assemblies regarded as a path category by declaring every map to be a fibration and the isomorphisms to be the equivalences. The construction in this section is similar, but different: indeed, in ${\rm Ex}({\bf PAsm})$ not every fibration would be a fibration of sets, as happens for $\mathbb{EFF}$ (see \refprop{inEFFeverythingset} below).
\end{rema}

\section{Function spaces in $\mathbb{EFF}$}

The purpose of this section is to show that $\mathbb{EFF}$ has homotopy $\Pi$-types. We will only give the constructions here, as the verifications are both straightforward and tedious.

\begin{prop}{EFFhasexp}
The path category $\mathbb{EFF}$ has homotopy exponentials.
\end{prop}
\begin{proof}
If $B = (B, \beta, \mathcal{B})$ and $A = (A, \alpha, \mathcal{A})$ are two objects in $\mathbb{EFF}$, then the homotopy exponential $A^B$ can be constructed as follows:
\begin{itemize}
\item The elements of the underlying set are triples consisting of a morphism $f: B \to A$ and trackings for both its 0- and 1-cells.
\item Realized by a code for the pair consisting of the tracking for the 0- and 1-cells.
\item And the 1-cells between $f$ and $g$ are coded homotopies between $f$ and $g$.
\end{itemize}
This completes the construction.
\end{proof}

\begin{rema}{onhexpinEFF} The construction in the previous proposition does not yield a genuine exponential in $\mathbb{EFF}$, for more than one reason. First of all, there is some arbitrariness in how the evaluation morphism $A^B \times B \to A$ is defined on 1-cells. In addition, since trackings are not unique, there is no canonical choice for the exponential transpose.
\end{rema}

\begin{prop}{EFFhasPi}
The path category $\mathbb{EFF}$ has homotopy $\Pi$-types.
\end{prop}
\begin{proof}
If $(A, \alpha, \mathcal{A})$ is an object in $\mathbb{EFF}$ and $a \in A$, then this induces an arrow $a: 1 \to A$ sending the unique 0-cell in 1 to $a$ and the unique 1-cell in 1 to $1_a$. In addition, if $\pi \in \mathcal{A}(a,a')$, then this induces a homotopy between the parallel arrows $a, a': 1 \to A$, which we will also denote by $\pi$. So it follows from \refprop{transportisanequivalence} that if $f: (B, \beta, \mathcal{B}) \to (A, \alpha, \mathcal{A})$ is a fibration and $\pi \in \mathcal{A}(a,a')$, then $\pi$ induces an equivalence $\Gamma^B_\pi: B_a \to B_{a'}$, where $B_a$ and $B_{a'}$ are the pullbacks of $f$ along $a: 1 \to A$ and $a': 1 \to A$, respectively.

If $g: (C, \gamma, \mathcal{C}) \to (B, \beta, \mathcal{B})$ is another fibration, then the fibration
\[ \Pi_f(g): (D, \delta, \mathcal{D}) \to (A, \alpha, \mathcal{A}) \]
is constructed as follows. We define $(D, \delta, \mathcal{D})$ by:
\begin{itemize}
\item Elements of $D$ consist of an element $a \in A$, a section $s: B_a \to C$ of $g$ and trackings for the 0- and 1-cells of $s$.
\item A realizer of such an element is a coded triple consisting of the realizer of $a \in A$ and the trackings for the 0- and 1-cells of $s$.
\item A 1-cell between $(a, s)$ and $(a', s')$ consists of a code for a pair containing an element $\pi: a \to a'$ and a coded homotopy between $s' \circ \Gamma^B_\pi$ and $\Gamma^C_\pi \circ s$.
\end{itemize}
In addition, the map $\Pi_f(g)$ is obtained by projecting on $a$ and $\pi$. Lengthy verifications are now in order.
\end{proof}

\section{Discrete fibrations in $\mathbb{EFF}$}

Within the path category $\mathbb{EFF}$ one can define an impredicative class of small fibrations given by the discrete fibrations.

\begin{defi}{discretefibrationsinEFF} A fibration $f: (B, \beta, \mathcal{B}) \to (A, \alpha, \mathcal{A})$ is a \emph{standard discrete fibration} if for elements $b, b' \in B$ we have:
\begin{quote}
if $f(b) = f(b')$ and $\beta(b) = \beta(b')$, then $b = b'$.
\end{quote}
A fibration $f$ is \emph{discrete} if $f$ can be written as $gw$ with $w$ being an equivalence and $g$ a standard discrete fibration.
\end{defi}

\begin{rema}{onalternativedeffordiscrfibr}
We could also have said that a fibration $f$ is discrete if there a standard discrete fibration $g$ and an equivalence $w$ such that $f \simeq gw$, or $fw \simeq g$, or $fw = g$. By arguments as in \refprop{hlevelsstableunderpbk} one can show that all these definitions are equivalent.
\end{rema}

\begin{prop}{discrfibrimprclassofsmallmaps} The discrete fibrations form an impredicative class of small fibrations in that:
\begin{enumerate}
\item every isomorphism is a discrete fibration, 
\item discrete fibrations are stable under homotopy pullback,
\item discrete fibrations are closed under composition, and
\item discrete fibrations are closed under $\Pi_f$ with $f$ a (not necessarily discrete) fibration.
\end{enumerate}
\end{prop}
\begin{proof}
It is easy to see that every isomorphism is a standard discrete fibration, that standard discrete fibrations are stable under pullback, closed under composition and that $\Pi_f(g)$ as constructed in \refprop{EFFhasPi} is a standard discrete fibration whenever $f$ and $g$ are fibrations and $g$ is standard discrete. From this the statement follows.
\end{proof}

\begin{rema}{otherdiscretethings}
\begin{enumerate}
\item It also follows from the construction of path objects in $\mathbb{EFF}$ as in \reftheo{EFFisapathcat} that $(s, t): PX \to X \times X$ is discrete for any object $X$. (In fact, we have that $P_I(X) \to X \times_I X$ is discrete for any fibration $X \to I$.)
\item One also easily verifies that $\mathbb{EFF}$ has homotopy sums (as in \cite{bergmoerdijk18}) and that for any two discrete fibrations $f: B \to A$ and $g: D \to C$, the map $f + g: B + D \to A + C$ is again discrete, and that for any two discrete fibrations $f: B \to A$ and $g: C \to A$ the map $[f,g]: B + C \to A$ is again discrete.
\item We can call an object $A$ discrete if the map $A \to 1$ is a discrete fibration. So, basically, an object $(A, \alpha, \mathcal{A})$ is discrete if $\alpha$ is injective, or if it is homotopic to such an object (see also the definition of the category $\mathbb{DISC}$ in Section 11 below). Examples of such discrete objects are the terminal object, its finite coproducts (0,1,2), as well as the homotopy natural numbers object $(N, \nu, \mathcal{N})$ with:
\begin{eqnarray*}
N & = & \mathbb{N} \\
\nu(n) & = & \{ n \} \\
\mathcal{N}(n,n) & = &  \{ 0 \} \\
\mathcal{N}(n, m) & = & \emptyset \mbox{ whenever } n \not= m.
\end{eqnarray*}
\end{enumerate}
\end{rema}

According to \cite{hylandetal90}, Peter Freyd suggested defining discrete maps in $\Eff$ by an internal orthogonality condition. It turns out that the same is possible in $\mathbb{EFF}$. First of all, let $\mathbb{J} = (J, j, \mathcal{J})$ be the object in $\mathbb{EFF}$ with
\begin{eqnarray*}
J & = & \{ 0, 1 \} \\
j(i) & = & 0 \\
\mathcal{J}(i,i) & = &  \{ 0 \} \\
\mathcal{J}(i, k) & = & \emptyset \mbox{ whenever } i \not= k.
\end{eqnarray*}
If $(A, \alpha, \mathcal{A})$ is any object in $\mathbb{EFF}$, then the homotopy exponential $A^\mathbb{J}$ can be computed as follows:
\begin{enumerate}
 \item The underlying set of $A^\mathbb{J}$ consists of pairs of elements $(a_0,a_1)$ in $A$ with $\alpha(a_0) = \alpha(a_1)$.
 \item The realizer of such a pair $(a_0,a_1)$ is $\alpha(a_0) = \alpha(a_1)$.
 \item The set $\mathcal{A}((a_0,a_1),(b_0,b_1))$ is the intersection of $\mathcal{A}(a_0,b_0)$ and $\mathcal{A}(a_1,b_1)$.
\end{enumerate}
Moreover, there is an obvious diagonal map $A \to \mathcal{A}^\mathbb{J}$ which copies the 0- and 1-cells. If $f: B \to A$ is any map, then this induces a (strictly) commuting square
\begin{center}
\begin{tikzcd}
B \ar[r] \ar[d, "f"'] & B^\mathbb{J} \ar[d, "f^\mathbb{J}"] \\
A \ar[r] & A^\mathbb{J}, 
\end{tikzcd}
\end{center}
and if $f$ is a fibration, then so is $f^\mathbb{J}$.

\begin{prop}{discretealafreydinEFF} {\bf (ZFC)} The following are equivalent for a fibration $f: B \to A$:
\begin{enumerate}
 \item $f$ is discrete.
 \item The square above is a homotopy pullback.
 \item There is a computable function which given two elements $b_0,b_1 \in B$ with $f(b_0) = f(b_1)$ and $n = \beta(b_0) = \beta(b_1)$ computes an element in $\mathcal{B}(b_0,b_1)$ from $n$.
\end{enumerate}
\end{prop}
\begin{proof}
(1) $\Rightarrow$ (2): If $f$ is a standard discrete fibration, then the square above is actually a genuine pullback. Since the statement in (2) is homotopy invariant, the implication follows.

(2) $\Rightarrow$ (3): The statement in (2) means that the diagonal map
\[ d: B \to \{ (b_0,b_1) \, : \, f(b_0) = f(b_1), \beta(b_0) = \beta(b_1) \} \]
is an equivalence. Since we have $\pi d = 1$ where $\pi$ is the projection on the first coordinate, $d$ is an equivalence precisely when $d \pi \simeq 1$. Unwinding this gives the statement in (3).

(3) $\Rightarrow$ (1): Consider the equivalence relation $\sim$ on $B$ defined by
\[ b_0 \sim b_1 \mbox{ if and only if } f(b_0) = f(b_1) \mbox{ and } \beta(b_0) = \beta(b_1), \]
and choose $B'$ to be a subset of $B$ which contains precisely one element from each equivalence class. The inclusion of $B'$ in $B$ determines a map $(B',\beta,\mathcal{B}) \to (B, \beta, \mathcal{B})$ with $\beta$ and $\mathcal{B}$ on $B'$ just being the restrictions to $B'$. The statement in (3) implies that this map is an equivalence, while its composition with $f$ is a standard discrete fibration by construction.
\end{proof}

\begin{rema}{notinterval} Superficially, $\mathbb{J}$ may look a bit like an interval, but it is really rather different from the object which should be considered as the interval object in $\mathbb{EFF}$, which is $\mathbb{I} = (I, \iota, \mathcal{I})$ with
\begin{eqnarray*}
I & = & \{ 0, 1 \} \\
\iota(i) & = & i \\
\mathcal{I}(i,j) & = &  \{ 0 \}.
\end{eqnarray*}
Indeed, for this object we have that $\mathbb{I} \to 1$ is a trivial fibration and that $PX \simeq X^\mathbb{I}$, while $\mathbb{J} \to 1$ is far from a trivial fibration and we have $X \simeq PX \simeq X^\mathbb{J}$ only if $X$ is discrete. Since $\mathbb{J}$ is a version of $\nabla(2)$, this also shows that our work here is rather different in spirit from that in \cite{vanoosten15} and \cite{fruminvandenberg17}. Another difference is that these papers seek to define homotopy-theoretic structures on the effective topos and do not regard it as a homotopy-theoretic quotient of some other category.
\end{rema}

\section{Propositions in $\mathbb{EFF}$}

It is not clear to me if the class of discrete fibrations has a representation: indeed, this is the same problem as the one whether the class of discrete maps (defined locally) has a representation in $\Eff$, which is still open. It will be possible to show, using a standard argument, that no representation of the class of discrete fibrations in $\mathbb{EFF}$ can be univalent  (see \refcoro{nounivreprfordiscrfibr} below). 

But first we will show that one can obtain a univalent representation if we restrict to the discrete propositional fibrations. As a first step in this direction, we will explicitly construct propositional truncations in $\mathbb{EFF}$.

\begin{prop}{existofproptrinEFF}
(Existence of propositional truncation) Every fibration $f: B \to A$ factors as a map $g: B \to C$ followed by a fibration of propositions $h: C \to A$, in a ``universal'' way: if $g': B \to C'$ and $h': C' \to A$ is another such factorisation, there is a map $d: C \to C'$ over $A$ with $dg \simeq_A g'$, with the map $d$ itself being unique up to fibrewise homotopy over $A$ with this property.
\end{prop}
\begin{proof}
Let me just give the construction: if $f: (B, \beta, \mathcal{B}) \to (A, \alpha, \mathcal{A})$ is a fibration, then we define $C := (B, \beta, \mathcal{C})$ with $\mathcal{C}(b,b') = \mathcal{A}(fb,fb')$. It is now easy to see that $C \to C \times_A C$ is an equivalence, turning $C \to A$ into a propositional fibration; the remaining verifications are also easy and left to the reader. 
\end{proof}

Note that this means that, up to equivalence, any propositional fibration \[ f: (B, \beta, \mathcal{B}) \to (A, \alpha, \mathcal{A}) \] can be assumed to satisfy $\mathcal{B}(b,b') = \mathcal{A}(fb,fb')$; if $f$ is also discrete, we may even assume that $B \subseteq A \times \mathbb{N}$ with $f$ being the projection on the first and $\beta$ the projection on the second coordinate (the reason for this is that propositional truncation as in the proof of the previous proposition preserves standard discrete fibrations). This is quite helpful when trying to prove:

\begin{theo}{discrpropsclassified}
The class of discrete propositional fibrations in $\mathbb{EFF}$ is an impredicative class of small fibrations with a univalent representation.
\end{theo}
\begin{proof}
It follows from \reftheo{hlevelsandPi} and \refprop{discrfibrimprclassofsmallmaps} that the discrete propositional fibrations form an impredicative class of small fibrations, so it remains to construct a univalent representation $\pi: E \to U$. We do this as follows:
\begin{itemize}
\item $U$ consists of subsets $X$ of $\mathbb{N}$
\item any such $X$ is realized by 0, 
\item a 1-cell between $X$ and $Y$ in $U$ consists of realizers $r: X \to Y$ and $s: Y \to X$ (that is, $r$ is a G\"odel number of a partial recursive function which is defined on every element in $X$ and which on these elements in $X$ outputs an element in $Y$; similarly for $s$),
\end{itemize}
while:
\begin{itemize}
\item $E$ consists of pairs $(X,x)$ with $X$ a subset of $\mathbb{N}$ and $x \in X$,
\item the pair $(X,x)$ is realized by $x$,
\item a 1-cell between $(X,x)$ and $(Y,y)$ is the same as a 1-cell between $X$ and $Y$.
\end{itemize}
This is clearly a discrete propositional fibration, and to see that it is a univalent representation, consider a map $f: (B, \beta, \mathcal{B}) \to (A, \alpha, \mathcal{A})$ and assume that $\mathcal{B}(b,b') = \mathcal{A}(fb,fb')$ and $B \subseteq A \times \mathbb{N}$ with $f$ being the projection on the first and $\beta$ the projection on the second coordinate. This gives rise to map $k: A \to U$ which is obtained by sending $a \in A$ to $B_a = \{ n \in \mathbb{N} \, : \, (a,n) \in B \}$ and by sending $\pi \in \mathcal{A}(a,a')$ to a tracking of the equivalence $\Gamma^B_\pi: B_a \to B_{a'}$. The remaining verifications are left to the reader.
\end{proof}

\begin{prop}{inEFFeverythingset}
In $\mathbb{EFF}$ every fibration is a fibration of sets.
\end{prop}
\begin{proof}
If $f: (B, \beta, \mathcal{B}) \to (A, \alpha, \mathcal{A})$ is a fibration, then we can factor $B \to B \times_A B$ as an equivalence $B \to P$ followed by a fibration $P \to B \times_A B$ with $P = (P, \pi, \mathcal{P})$ given by:
\begin{eqnarray*}
P & = & \{ (b \in B, b' \in B, \pi \in \mathcal{B}(b,b') \, : \, f(b) = f(b') \} \\
\pi(b,b',\pi) & = & \langle \beta(b),\beta(b'),\pi \rangle \\
\mathcal{P}((b_0,b'_0,\pi_0), (b_1,b'_1,\pi_1)) & = & \{ \langle n, n' \rangle \, : \, n \in \mathcal{B}(b_0,b_1), n' \in \mathcal{B}(b_0',b_1') \}
\end{eqnarray*}
Since $P \to B \times_A B$ is a propositional fibration, $f$ is a fibration of sets.
\end{proof}

\begin{coro}{nounivreprfordiscrfibr}
The class of discrete fibrations in $\mathbb{EFF}$ does not have a univalent representation.
\end{coro}
\begin{proof} If a class of small fibrations has a univalent representation $\pi: E \to U$ with $U$ a set, then every small object $A$ has, up to homotopy, precisely one self-equivalence. Indeed, an object with more than one self-equivalence would give rise to a map $1 \to U$ which is homotopic to itself in more than one way, which is impossible if $U$ is a set. 

So because $2$ is a discrete object which is equivalent to itself in two different ways, the class of discrete fibrations cannot have a univalent representation in $\mathbb{EFF}$. (This is, in fact, a well-known argument: see \cite{kraussattler15}.)
\end{proof}

Finally, let us point out that, assuming the axiom of choice in the metatheory, we can prove that the discrete (propositional) fibrations satisfy resizing.

\begin{prop}{resizinginEFF}
{\bf (ZFC)} In $\mathbb{EFF}$ propositional fibrations are discrete. Therefore the the discrete (propositional) fibrations in $\mathbb{EFF}$ are a class of small fibrations satisfying resizing.
\end{prop}
\begin{proof}
Suppose $f: (B, \beta, \mathcal{B}) \to (A, \alpha, \mathcal{A})$ is propositional, and let \[ C = \{ (a, n) : (\exists b) \, f(b) = a, \beta(b) = n \}, \] with $\gamma(a, n) = n$ and $\mathcal{C}((a,n),(a', n')) = \mathcal{A}(a, a')$. Clearly, the projection $\pi: C \to A$ is a discrete propositional fibration. There is an obvious morphism $B \to C$ over $A$, and, using AC, there is also a morphism $C \to B$ over $A$. Since $f$ and $\pi$ are propositional, the maps going back and forth between $B$ and $C$ are each other's homotopy inverses, so $f$ is discrete as well.
\end{proof}

\section{A remark on Church's Thesis in $\mathbb{EFF}$}

A classical result in the metamathematics of constructivism says that extensional Heyting arithmetic in finite types with the axiom of choice for finite types proves the negation of Church's Thesis (saying that any function $f: \mathbb{N} \to \mathbb{N}$ is computable) (see, for instance, \cite[Theorem 9.6.8.(i)]{troelstravandalen88}). This means in particular that if we work in Martin-L\"of Type Theory with function extensionality, then this theory will refute Church's Thesis in the following form:
\[ \Pi f: \mathbb{N} \to \mathbb{N}. \Sigma n \in \mathbb{N}. \mbox{$n$ is a G\"odel number for a Turing machine computing $f$}. \]
However, if we use propositional truncation $\Vert \ldots \Vert$ to turn the $\Sigma$ into $\exists$ (to use the terminology from Homotopy Type Theory) and we read Church's Thesis as:
\[ \Pi f: \mathbb{N} \to \mathbb{N}. \Vert \Sigma n \in \mathbb{N}. \mbox{$n$ is a G\"odel number for a Turing machine computing $f$} \Vert, \]
then this statement holds in $\mathbb{EFF}$. One way of seeing this starts by observing that if $P$ is the functor from the proof of \refprop{hoisEff} and $f: B \to A$ is a fibration in $\mathbb{EFF}$, then $Pf$ is a monomorphism precisely when $f$ is propositional (this should be clear from the description of the functor $P$, the discussion on page 67 of \cite{vanoosten08} and \refprop{existofproptrinEFF}). Therefore for any object $A$ in $\mathbb{EFF}$, the functor $P$ sets up an order isomorphism between the subobject lattice on $PA$ in $\Eff$ and the poset reflection of the full subcategory of the slice category $\mathbb{EFF}/A$ on the propositional fibrations. This means that what holds in the logic of $\mathbb{EFF}$, when we interpret statements using propositions in the sense of Homotopy Type Theory (as explained in Sections 3.6 and 3.7 in \cite{univalent13}), is precisely what holds in the internal logic of $\Eff$. Since Church's Thesis holds there, the same is true for $\mathbb{EFF}$. (See also \cite{escardoxu15}, especially Section 1.3.)

\section{The path category $\mathbb{EFF}_1$}

A univalent universe containing sets like $2, \mathbb{N}$ and $\mathbb{N} \to \mathbb{N}$ must have an hlevel at least one: that is, it must at least be a groupoid. This means that if we want a category like $\mathbb{EFF}$ in which there is a univalent representation classifying discrete sets, we have to consider objects with an additional level of complexity. Indeed, for the next step we need to consider tricategories which are locally codiscrete in the sense that any two parallel 2-cells are connected by a unique 3-cell. This sounds quite complicated, but it is not too bad: we have seen that locally codiscrete bigroupoids are in a sense easier than groupoids, and in the same way locally codiscrete tricategories are simpler objects than bicategories.

We define the following category $\mathbb{EFF}_1$:
\begin{itemize}
\item Objects consist of:
\begin{enumerate}
\item A set $A$.
\item A function $\alpha: A \to \mathbb{N}$.
\item For each pair of elements $a, a' \in A$ a set of natural numbers (1-cells) $\mathcal{A}(a,a')$.  Instead of $\pi \in \mathcal{A}(a,a')$ we will often write $\pi: a \to a'$.
\item For each pair of elements $\pi, \pi' \in \mathcal{A}(a,a')$ a set of natural numbers (2-cells) $\mathcal{A}(\pi,\pi')$. Instead of $n \in \mathcal{A}(\pi,\pi')$ we will often write $n: \pi \Rightarrow \pi'$. 
\item An effective function picking for each $a \in A$ an element $1_a \in \mathcal{A}(a,a)$.
\item An effective function picking for each $a, a' \in A$ and $\pi \in \mathcal{A}(a,a')$ an element $\pi^{-1} \in \mathcal{A}(a',a)$.
\item An effective function picking for each $a,a',a'' \in A$ and $\pi \in \mathcal{A}(a,a')$ and $\rho \in \mathcal{A}(a',a'')$ an element $\rho \circ \pi \in \mathcal{A}(a,a'')$.
\item Together with 2-cells $1_b \circ \pi  \Rightarrow \pi$ and  $\pi \circ 1_a \Rightarrow \pi$ (which can be computed from $a,b,\pi$).
\item Together with 2-cells  $\pi^{-1} \circ \pi \Rightarrow 1$ and $\pi \circ \pi^{-1} \Rightarrow 1$ (which can be computed from the same data).
\item Together with 2-cells $(\sigma \circ \rho) \circ \pi \Rightarrow \sigma \circ (\rho \circ \pi)$ (which can be computed from $\pi,\rho,\sigma$ and realizers for the domains and codomains).
\item For each $\pi \in \mathcal{A}(a,a')$ a 2-cell $1_{\pi} \in \mathcal{A}(\pi,\pi)$ (which can be computed from $\pi$ and realizers for $a$ and $a'$).
\item For each pair of 2-cells $n \in \mathcal{A}(\pi,\rho)$ and $m \in \mathcal{A}(\rho,\sigma)$ a 2-cell $m \circ_1 n \in \mathcal{A}(\pi,\sigma)$ (which can be computed from $n, m, \pi, \rho, \sigma$ and their domains and codomains).
\item For each 2-cell $n \in \mathcal{A}(\pi,\rho)$ a 2-cell $n^{-1} \in \mathcal{A}(\rho,\pi)$ (which can again be computed from the relevant data).
\item For each pair of 2-cells $n \in \mathcal{A}(\pi,\rho)$ and $m \in \mathcal{A}(\pi',\rho')$ a 2-cell $m \circ_0 n \in \mathcal{A}(\pi' \circ \pi,\rho' \circ \rho)$ (which can also be computed from the relevant data).
\end{enumerate}
\item Morphisms $f: B \to A$ consist of: 
\begin{enumerate}
\item A function $f: B \to A$.
\item For each $b_1,b_2 \in B$ a function $f_{b_1,b_2}: \mathcal{B}(b_1,b_2) \to \mathcal{A}(fb_1,fb_2)$, as well as
\item For each $b_1,b_2 \in B$ and $\pi,\rho \in \mathcal{B}(b_1,b_2)$ a function $f_{\pi,\rho}:\mathcal{B}(\pi,\rho) \to \mathcal{A}(f\pi,f\rho)$,
\end{enumerate}
such that:
\begin{itemize}
\item A realizer for $f(b)$ can be computed from a realizer for $b$.
\item $f_{b_1,b_2}(\pi)$ can be computed from realizers for $b_1,b_2$ and $\pi$.
\item $f_{\pi,\rho}(n)$ can be computed from $n, \pi$ and $\rho$ and realizers for the domain and codomain of both $\pi$ and $\rho$.
\item An element in $\mathcal{A}(f(1_b),1_{fb})$ can be computed from a realizer for $b$.
\item An element in $\mathcal{A}(f(\rho \circ \pi),f\rho \circ f\pi)$ can be computed from realizers for $\rho,\pi$ and domains and codomains for $\rho$ and $\pi$.
\end{itemize}
\end{itemize}

\begin{rema}{lasttwoconditons}
It may look odd to not make the last two items (the ``proof terms for functoriality'') part of the identity of a morphism. However, since we lack unit and associativity laws for 2-cells in the definition of an object in $\mathbb{EFF}_1$ adding such would no longer result in a category.
\end{rema}

I will continue to denote objects by $(A, \alpha, \mathcal{A})$.

\begin{defi}{fibrationinEFF1}
A morphism $f: (B, \beta, \mathcal{B}) \to (A, \alpha, \mathcal{A})$ in $\mathbb{EFF}_1$ is called a \emph{fibration} if
\begin{itemize}
\item for every $b \in B$, $a \in A$ and $\pi \in \mathcal{A}(fb, a)$ one can effectively find elements $b' \in B$ with $f(b') = a$ and $\pi' \in \mathcal{B}(b, b')$ with $f(\pi') = \pi$.
\item for any $\pi, \pi' \in \mathcal{A}(fb,fb')$, $n \in \mathcal{A}(\pi,\pi')$ and $\rho \in \mathcal{B}(b,b')$ such that $f(\rho) = \pi$, one can compute $\rho' \in \mathcal{B}(b,b')$ and $m \in \mathcal{B}(\rho,\rho')$ with $f(\rho') = \pi'$ and $f(m) = n$.
\item given a 2-cell $m \in \mathcal{B}(\pi,\pi')$ and a 2-cell $n \in \mathcal{A}(f\pi,f\pi')$ one can compute a 2-cell $m' \in \mathcal{B}(\pi,\pi')$ with $f(m') = n$.   
\end{itemize}
\end{defi}

\begin{lemm}{easyfibr}
Isomorphisms are fibrations, fibrations are closed under composition and the unique map to the terminal object is always a fibration.
\end{lemm}

\begin{lemm}{onpbks}
Fibrations are stable under pullback.
\end{lemm}
\begin{proof}
Let $f: B \to A$ be a fibration and $g: C \to A$ be arbitrary. The pullback is computed as follows:
\begin{itemize}
\item The underlying set is $D := C \times_A B$.
\item A realizer of $(c, b)$ is the pair consisting of $\gamma(c)$ and $\beta(b)$.
\item 1-cells $(c,b) \to (c',b')$ are pairs of 1-cells $\pi: c \to c'$ and $\rho: b \to b'$ with $g(\pi) = f(\rho)$.
\item 2-cells $(\pi,\rho) \Rightarrow (\pi',\rho')$ are pairs of 2-cells $n: \pi \Rightarrow \pi'$ and $m: \rho \Rightarrow \rho'$ with $g(n) = f(m)$.
\end{itemize}
The main difficulty here is not to show the universal property, but to prove that it is an object of $\mathbb{EFF}_1$. As in \reftheo{EFFisapathcat}, one defines identities or composition on the $C$-component of $D$ just as in $C$, whereas for the $B$-component of $D$ one uses the fact that $f$ is a fibration to choose something in $B$ which is both homotopic to the identity or composition of $B$ and which has the same $A$-image as the object chosen in the $C$-component.
\end{proof}

\begin{defi}{homotopyinEFF1}
A \emph{homotopy} between two morphisms $f, g: (B, \beta, \mathcal{B}) \to (A, \alpha, \mathcal{A})$ in $\mathbb{EFF}_1$ consists of:
\begin{enumerate}
\item a function $H_1$ picking for every $b \in B$ an element in $H_1(b) \in \mathcal{A}(fb, gb)$.
\item a function $H_2$ picking for each $\pi \in \mathcal{B}(b,b')$ a filler $H_2(\pi)$ for the square
\begin{center}
\begin{tikzcd}
 fb \ar[r, "f(\pi)"] \ar[d, "H_1(b)"'] & fb' \ar[d, "H_1(b')"] \\
gb \ar[r, "g(\pi)"'] & gb' 
\end{tikzcd}
\end{center}
\end{enumerate}
such that:
\begin{itemize}
\item[--] $H_1(b)$ can be computed from $b$.
\item[--] $H_2(\pi)$ can be computed from $b,b'$ and $\pi$.
\end{itemize}
If $H$ is such a homotopy, I will write $H: f \simeq g$. If such a homotopy $H$ exists, I will call $f$ and $g$ homotopic and write $f \simeq g$. A \emph{coded homotopy} is a pair consisting of G\"odel numbers for the trackings of $H_1$ and $H_2$. If a map has a homotopy inverse, it will be called a \emph{(homotopy) equivalence}.
\end{defi}

Note that the homotopy relation does not take into account the action of morphisms on 2-cells.

\begin{lemm}{2outof6}
The homotopy relation is a congruence and hence isomorphisms are equivalences and equivalences satisfy 6-for-2.
\end{lemm}

\begin{lemm}{pathobjects} Every diagonal $A \to A \times A$ factors as an equivalence followed by a fibration.
\end{lemm}
\begin{proof}
The \emph{path object} on $A = (A, \alpha, \mathcal{A})$ is the object $PA = (P, \pi, \mathcal{P})$, where:
\begin{enumerate}
\item $P$ consists of triples $(a, b, \rho)$ with $a, b \in A$ and $\rho \in \mathcal{A}(a,b)$,
\item $\pi(a, b, \rho)$ is a triple consisting of the realizers of $a$ and $b$, as well as the natural number $\rho$, and
\item $\mathcal{P}((a,b, \rho), (a',b',\rho'))$ consists of pairs of maps $\mu: a \to a'$ and $\nu: b \to b'$, together with a natural number $n: \nu \rho \Rightarrow \rho'\mu$. 
\item 2-cells between parallel $(\mu,\nu,n)$ and $(\mu',\nu',n')$ consist of pairs of 2-cells $\mu \Rightarrow \mu'$ and $\nu \Rightarrow \nu'$.
\end{enumerate}
In addition, there is an obvious factorisation of the diagonal 
\begin{center}
\begin{tikzcd}
A \ar[r, "r"] & PA \ar[r,"{(s,t)}"] & A \times A,
\end{tikzcd}
\end{center}
in which the first map is an equivalence and the second a fibration.
\end{proof}

If $H, K: f \simeq g$ are homotopies between $f, g: B \to A$, let us write $M: H \sim K$ for a function $M$ which computes for each $b \in B$ a 2-cell in $A$ between $H_b: fb \to gb$ and $K_b: fb \to gb$.

The following proposition is a variation on a standard result from homotopy theory and higher category theory (see \cite[Theorem 4.2.3]{univalent13}).

\begin{prop}{adjequiv}
Suppose $f: A \to B$ and $g: B \to A$ are morphisms with homotopies $\eta: 1 \simeq gf$ and $\varepsilon: fg \simeq 1$. Then we can modify $\eta$ to a homotopy $\eta': 1 \simeq gf$ for which there are
\[ M: \varepsilon_f \circ f\eta' \sim 1_f \mbox{ and } N: g\varepsilon \circ \eta'_g \sim 1_g. \]
(A similar statement holds for $\varepsilon$.)
\end{prop}
\begin{proof}
Let us first observe that if $k: A \to A$ is an endomorphism and $H: 1 \simeq k$ is a homotopy, then there are two homotopies $k \simeq kk$, to wit $H_k$ and $kH$. There is however a homotopy $M: H_k \sim kH$ between those. The reason is that
\diag{ a \ar[r]^{H_a} \ar[d]_{H_a} & ka \ar[d]^{H_{ka}} \\
ka \ar[r]_{k(H_a)} & kka }
is a naturality square and $H_a$ is an isomorphism.

So let us define
\[ \eta'_a := \eta^{-1}_{gfa} \circ g(\varepsilon_{fa}^{-1}) \circ \eta_a. \]
Since $\varepsilon_f: fgf \simeq f$ is a natural transformation and $\eta_a: a \to gfa$ is a morphism, we have that
\diag{ fgfa \ar[d]_{\varepsilon_{fa}} \ar[r]^{fgf(\eta_a)} & fgfgfa \ar[d]^{\varepsilon_{fgfa}} \\
fa \ar[r]_{f\eta_a} & fgfa }
commutes up to some 2-cell. Since we also have 2-cells $fgf\eta_a \Rightarrow f(\eta_{gfa})$ and $\varepsilon_{fgfa} \Rightarrow fg(\varepsilon_{fa})$ by the previous observation, we find $M: \varepsilon_f \circ f(\eta') \sim 1_f$.

Dually, we have the natural transformation $\eta_g: g \simeq gfg$ and the 1-cell $\varepsilon_b: fgb \to b$ leading to a 2-cell filling:
\diag{ gfgb \ar[r]^{\eta_{gfgb}} \ar[d]_{g\varepsilon_b} & gfgfgb \ar[d]^{gfg\varepsilon_b} \\
gb \ar[r]_{\eta_{gb}} & gfgb. }
Since there is a 2-cell $gfg\varepsilon_b \Rightarrow g(\varepsilon_{fgb})$, we also obtain $N_b: g\varepsilon_b \circ \eta'_{gb} \sim 1_{gb}$.
\end{proof}

\begin{prop}{ontrivfibr}
A fibration $f: B \to A$ is trivial if and only if $f$ has a section $s: A \to B$ for which there are $H: 1_B \simeq sf$ and $M: fH \sim 1_f$.
\end{prop}
\begin{proof}
If such $s, H, M$ exist, $f$ is an equivalence, so it suffices to prove the other direction.

So assume $f: B \to A$ is a fibration which is also an equivalence. This means there are maps $g: A \to B$ and homotopies $\varepsilon: fg \simeq 1$ and $\eta: 1 \simeq gf$. By \refprop{adjequiv} we may assume that there is some $N: \varepsilon_f \circ f(\eta) \sim 1_f$.

Let  $a \in A$ be arbitrary. Since $\varepsilon_a: fga \simeq a$ and $f$ is a fibration, we can effectively find an element $sa \in B$ and a 1-cell $\tau_a: ga \to sa$ such that $f(\tau_a) = \varepsilon_a$. In addition, if $\pi: a \to a'$, then for $\rho' := \tau_{a'} \circ g(\pi) \circ \tau_a^{-1}: sa \to sa'$, we can compute a 2-cell $n: f(\rho') \simeq \pi$, so we can determine an element $s(\pi): sa \to sa'$ such that $fs(\pi) = \pi$ and a 2-cell:
\[ \tau_{\pi}: s(\pi) \circ \tau_a \Rightarrow \tau_{a'} \circ g(\pi). \]
Finally, if $n: \pi \Rightarrow \pi'$ in $A$, then $g$ yields a 2-cell $m': s(\pi) \Rightarrow s(\pi')$ whose $f$-image is parallel to $n$. So we find an element $s(n): s(\pi) \Rightarrow s(\pi')$ with $f(s(n)) = n$.

From this it follows that $s$ is a morphism and $\tau$ is a homotopy between $g$ and $s$. By construction, $s$ is a section of $f$. Finally, if $H = \tau_f \circ \eta$, then $fH \sim f(\tau_f) \circ f(\eta) = \varepsilon_f \circ f(\eta) \sim 1_f$, as desired.
\end{proof}

\begin{coro}{trivstableunderpbk}
Trivial fibrations have sections and are stable under pullback along arbitrary maps.
\end{coro}

We have shown:

\begin{theo}{EFF1pathcategory}
With the equivalences and fibrations as defined above, the category $\mathbb{EFF}_1$ carries the structure of a path category.
\end{theo}

\begin{rema}{rosoliniagain} The category $\mathbb{EFF}_1$ is similar to one considered by Pino Rosolini in a talk at the Wessex Theory Seminar in 2010; however, his category is actually a 2-category, not a path category, and also his goals were different.
\end{rema}

\section{Function spaces and function extensionality in $\mathbb{EFF}_1$}

We now show that $\mathbb{EFF}_1$ is a path category with homotopy $\Pi$-types. We will just outline the constructions and omit the lengthy verifications.

\begin{prop}{hexpinEFF1}
The path category $\mathbb{EFF}_1$ has homotopy exponentials.
\end{prop}
\begin{proof}
If $B = (B, \beta, \mathcal{B})$ and $A = (A, \alpha, \mathcal{A})$ are two objects, then we define the homotopy exponential $A^B$ to be the following:
\begin{itemize}
\item At the set level it consists of all homomorphism $f: B \to A$ including all the trackings for 0,1 and 2-cells.
\item A realizer for $f$ consists of trackings for its 0,1 and 2-cells (but there is no reason to include the proof terms for functoriality).
\item 1-cells between $f$ and $g$ are coded homotopies $H: f \simeq g$.
\item 2-cells between $H$ and $K$ are coded homotopies $M: H \sim K$.
\end{itemize}
Checking that this works is a lengthy but routine exercise.
\end{proof}

\begin{prop}{functionspacesinEFF1}
The path category $\mathbb{EFF}_1$ has homotopy $\Pi$-types.
\end{prop}
\begin{proof} If $(A, \alpha, \mathcal{A})$ is an object in $\mathbb{EFF}_1$ and $a \in A$, then there is a morphism $1 \to (A, \alpha, \mathcal{A})$, which we also call $a$, sending the unique 0-cell in 1 to $a$ and the higher cells in 1 to the identity arrows on $a$ and $1_a$. If $f: (B, \beta, \mathcal{B}) \to (A, \alpha, \mathcal{A})$ is a fibration, then we will write $B_a$ for the object in $\mathbb{EFF}_1$ obtained by pulling back $f$ along $a: 1 \to A$. Moreover, if $\pi: a \to a'$ is any 1-cell in $A$, then this induces a homotopy between $a,a': 1 \to A$, which we will also call $\pi$. By \refprop{transportisanequivalence} this induces a homotopy equivalence $\Gamma^B_{\pi}: B_a \to B_{a'}$. In addition, if $n: \pi \Rightarrow \pi'$ is any 2-cell in $A$, this induces a fibrewise homotopy $\Gamma_n$ between $\Gamma_{\pi}$ and $\Gamma_{\pi'}$ (see \refprop{basicpropoftransport}.(2)).

If $g: (C, \gamma, \mathcal{C}) \to (B, \beta, \mathcal{B})$ is another fibration, we define $\Pi_f(g)$ to be the obvious projection to $A$ from the following object:
\begin{itemize}
\item The set consists of elements $a \in A$ together with sections $s: B_a \to C$ of $g$ with the required tracking data for 0,1,2-cells.
\item A realizer for such an element consists of a realizer for $a \in A$ together with the tracking data for the 0,1,2-cells.
\item 1-cells between $(a, s)$ and $(a',s')$ consist of a 1-cell $\pi: a \to a'$ as well as a coded fibrewise homotopy $H: \Gamma_{\pi}^C \circ s \simeq s' \circ \Gamma^B_{\pi}: B_a \to C$.
\item 2-cells between those consist of a 2-cell $n: \pi \to \pi'$ in $A$ together with a coded homotopy $M: (s' \circ \Gamma_n^B) \circ H \sim H' \circ (\Gamma^C_n \circ s)$ between the homotopies $\Gamma_\pi^C \circ s \simeq s' \circ \Gamma_{\pi'}^B$.
\end{itemize}
Some very lengthy verifications are now in order.
\end{proof}

\section{Discrete fibrations in $\mathbb{EFF}_1$}

Within the path category $\mathbb{EFF}_1$ we can again isolate an impredicative class of small fibrations given by the discrete fibrations.

\begin{defi}{discretefibrationsinEFF1} A fibration $f: (B, \beta, \mathcal{B}) \to (A, \alpha, \mathcal{A})$ is a \emph{standard discrete fibration} if for elements $b, b' \in B$ we have:
\begin{quote}
if $f(b) = f(b')$ and $\beta(b) = \beta(b')$, then $b = b'$.
\end{quote}
A fibration $f$ is \emph{discrete} if $f$ can be written as $gw$ with $w$ being an equivalence and $g$ a standard discrete fibration.
\end{defi}

By the same argument as before, we obtain:

\begin{prop}{discrfibrimprclassofsmallmapsagain} The discrete fibrations form an impredicative class of small fibrations in that:
\begin{enumerate}
\item every isomorphism is a discrete fibration, 
\item discrete fibrations are stable under homotopy pullback,
\item discrete fibrations are closed under composition, and
\item discrete fibrations are closed under $\Pi_f$ with $f$ a (not necessarily discrete) fibration.
\end{enumerate}
\end{prop}

Also, it is possible to again characterise discrete fibrations by an orthogonality condition, in the following sense. Let $\mathbb{J}$ be the object in $\mathbb{EFF}_1$ with two 0-cells both having realizer $0$ and only identity 1- and 2-cells, also given by the natural number 0. If $f: B \to A$ is a fibration, then the natural way of computing the homotopy exponential $A^\mathbb{J}$ gives rise to a (strictly) commuting square
\begin{center}
\begin{tikzcd}
B \ar[r] \ar[d, "f"'] & B^\mathbb{J} \ar[d, "f^\mathbb{J}"] \\
A \ar[r] & A^\mathbb{J}, 
\end{tikzcd}
\end{center}
in which the horizontal arrows are the diagonals map and $f^\mathbb{J}$ is also a fibration.

\begin{prop}{discretealafreydinEFF1} {\bf (ZFC)} The following are equivalent for a fibration $f: B \to A$:
\begin{enumerate}
 \item $f$ is discrete.
 \item The square above is a homotopy pullback.
 \item There are computable functions $\varphi$ and $\psi$ such that for any two elements $b_0,b_1 \in B$ with $a = f(b_0) = f(b_1)$ and $n = \beta(b_0) = \beta(b_1)$ the element $\varphi(n)$ belongs to $\mathcal{B}(b_0,b_1)$ and $f(\varphi(n)) = 1_a$, and such that for elements $b_0,b_1,c_0,c_1 \in B$ with $f(b_0) = f(b_1)$, $f(c_0)=f(c_1)$, $n = \beta(b_0) = \beta(b_1), m = \beta(c_0) = \beta(c_1)$ and $\pi \in \mathcal{B}(b_0,c_0) \cap \mathcal{B}(b_1,c_1)$ we have that $\psi(n,m,\pi)$ is a 2-cell filling
\begin{center}
\begin{tikzcd}
b_0 \ar[r, "\pi"] \ar[d, "\varphi(n)"'] & c_0 \ar[d, "\varphi(m)"] \\
b_1 \ar[r, "\pi"'] & c_1. 
\end{tikzcd}
\end{center}
\end{enumerate}
\end{prop}

\section{Propositions in $\mathbb{EFF}_1$}

\begin{prop}{existofproptrinEFF1}
(Existence of propositional truncation) Every fibration $f: B \to A$ factors as a map $g: B \to C$ followed by a fibration of propositions $h: C \to A$, in a ``universal'' way: if $g': B \to C'$ and $h': C' \to A$ is another such factorisation, there is a map $d: C \to C'$ over $A$ with $dg \simeq_A g'$, with the map $d$ itself being unique up to fibrewise homotopy over $A$ with this property.
\end{prop}
\begin{proof}
Let me just give the construction: if $f: (B, \beta, \mathcal{B}) \to (A, \alpha, \mathcal{A})$ is a fibration, then we define $C := (B, \beta, \mathcal{C})$ with $\mathcal{C}(b,b') = \mathcal{A}(fb,fb')$ and $\mathcal{C}(\pi,\pi')=\mathcal{A}(f\pi,f\pi')$. 
\end{proof}

Note that this means that, up to equivalence, any propositional fibration \[ f: (B, \beta, \mathcal{B}) \to (A, \alpha, \mathcal{A}) \] can be assumed to satisfy $\mathcal{B}(b,b') = \mathcal{A}(fb,fb')$ and $\mathcal{B}(\pi,\pi') = \mathcal{A}(f\pi, f\pi')$; if $f$ is also discrete, we may also assume that $B \subseteq A \times \mathbb{N}$ with $f$ being the projection on the first and $\beta$ the projection on the second coordinate. 

\begin{prop}{discrpropsclassified}
In $\mathbb{EFF}_1$ the class of discrete propositional fibrations is an impredicative class of small fibrations with a univalent representation.
\end{prop}
\begin{proof}
Discrete propositional fibrations are classified by the following $\pi: E \to U$: 
\begin{itemize}
\item $U$ consists of subsets $A$ of $\mathbb{N}$, with every such realized by 0.
\item $\mathcal{U}(A,B)$ consists of a pair a realizers $(r: A \to B, s: B \to A)$.
\item Between any two such parallel 1-cells there exists a unique 2-cell.
\item $E$ consists of pairs $(A, a)$ with $A \subseteq \mathbb{N}$ and $a \in A$, realized by $a$.
\item $\mathcal{E}((A,a),(B,b))$ is just $\mathcal{U}(A,B)$.
\item And, again, any two parallel 1-cells are connected by a unique 2-cell.
\end{itemize}
The projection $\pi: E \to U$ is clearly a fibration of discrete propositions. We omit the verification that it is a univalent representation.
\end{proof}

\begin{prop}{inZFC}
{\bf (ZFC)} In $\mathbb{EFF}_1$ every propositional fibration is discrete. Hence the class of discrete fibrations satisfies propositional resizing.
\end{prop}
\begin{proof}
Suppose $f: (B, \beta, \mathcal{B}) \to (A, \alpha, \mathcal{A})$ is propositional, and let \[ C = \{ (a, n) : (\exists b) \, f(b) = a, \beta(b) = n \}, \] with $\gamma(a, n) = n$ and $\mathcal{C}((a,n),(a', n')) = \mathcal{A}(a, a')$, and also the 2-cells in $C$ being those in $A$. Clearly, the projection $\pi: C \to A$ is a discrete fibration of propositions. There is an obvious morphism $B \to C$ over $A$, and, using the axiom of choice, there is also a morphism $C \to B$ over $A$. Since $f$ and $\pi$ are propositional, the maps going back and forth between $B$ and $C$ are equivalences, so $f$ is discrete as well.
\end{proof}

\section{Sets in $\mathbb{EFF}_1$}

So far we have only seen properties of $\mathbb{EFF}_1$ which are also properties of $\mathbb{EFF}$. The main difference, however, is that the fibrations of discrete sets also form an impredicative class of small fibrations in $\mathbb{EFF}_1$ with a univalent representation. Before we show this, let us first show that $\mathbb{EFF}_1$ allows a form of 0-truncation.

\begin{prop}{existof0trinEFF2}
(Existence of 0-truncation) Every fibration $f: B \to A$ factors as a map $g: B \to C$ followed by a fibration of sets $h: C \to A$, in a ``universal'' way: if $g': B \to C'$ and $h': C' \to A$ is another such factorisation, there is a map $d: C \to C'$ over $A$ with $dg \simeq_A g'$, with the map $d$ itself being unique up to fibrewise homotopy over $A$ with this property.
\end{prop}
\begin{proof}
Let me again just give the construction: if $f: (B, \beta, \mathcal{B}) \to (A, \alpha, \mathcal{A})$ is a fibration, then we define $C := (B, \beta, \mathcal{C})$ with $\mathcal{C}(b,b') = \mathcal{B}(b,b')$ and $\mathcal{C}(\pi,\pi')=\mathcal{A}(f\pi,f\pi')$. 
\end{proof}

Note that this means that any fibration of sets $f: (B, \beta, \mathcal{B}) \to (A, \alpha, \mathcal{A})$ can be assumed to satisfy $\mathcal{B}(\pi,\pi') = \mathcal{A}(f\pi,f\pi')$; if $f$ is also discrete, we may even assume that $B \subseteq A \times \mathbb{N}$ with $f$ being the first and $\beta$ being the second projection. This means that the category of discrete sets in $\mathbb{EFF}_1$ looks essentially as follows.

Its objects are:
\begin{enumerate}
\item A subset $A$ of $\mathbb{N}$.
\item For each pair of elements $a,a' \in A$ a subset $\mathcal{A}(a,a')$ of $\mathbb{N}$.
\item A function which computes for any $a \in A$ an element in $\mathcal{A}(a,a)$.
\item A function which given $a,a'$ and $\pi \in \mathcal{A}(a,a')$ computes an element $\pi^{-1} \in \mathcal{A}(a'a)$.
\item A function which given $a,a',a'' \in A$ and $\pi \in \mathcal{A}(a,a'), \pi' \in \mathcal{A}(a',a'')$ computes an element $\pi' \circ \pi \in \mathcal{A}(a,a'')$.
\end{enumerate}
We will often just write $(A, \mathcal{A})$ for these objects. A morphism $f: (B, \mathcal{B}) \to (A, \mathcal{A})$ between such objects consists of:
\begin{enumerate}
\item A computable function $f: B \to A$.
\item For each $b, b' \in B$ a function $f_{(b, b')}: \mathcal{B}(b,b') \to \mathcal{A}(fb, fb')$ such that $f_{(b,b')}(\pi)$ can be computed from $b,b'$ and $\pi$.
\end{enumerate}

We will denote this category by $\mathbb{DISC}$. Note that $\mathbb{DISC}$ is a small category and it is also, essentially, the subcategory of discrete objects in $\mathbb{EFF}$. In fact, because the map $PA \to A \times A$ is always discrete in $\mathbb{EFF}$, the category $\mathbb{DISC}$ inherits a path category structure from $\mathbb{EFF}$. 

\begin{theo}{discrsetsclassifiedinEFF2}
In $\mathbb{EFF}_1$ the fibrations of discrete sets form an impredicative class of small fibrations with a univalent representation.
\end{theo}
\begin{proof}
It follows from \refprop{discrfibrimprclassofsmallmapsagain} and \reftheo{hlevelsandPi} that the fibrations of discrete sets form an impredicative class of small fibrations in $\mathbb{EFF}_1$. That means that it remains to construct a univalent representation $\pi: E \to U$.
\begin{itemize}
\item The elements of $U$ are the objects $(A, \mathcal{A})$ of $\mathbb{DISC}$.
\item A realizer for any such object $(A, \mathcal{A})$ is 0.
\item The 1-cells are coded homotopy equivalences $(f,g,H, K)$ (in the sense of $\mathbb{DISC}$ and $\mathbb{EFF}$). 
\item The 2-cells between $(f,g,H,K)$ and $(f',g',H',K')$ consist of homotopies $U: f \simeq f'$.
\end{itemize}
In addition:
\begin{itemize}
\item The elements of $E$ are $(A, \mathcal{A}, a)$, where $(A, \mathcal{A})$ is an object of $\mathbb{DISC}$ and $a \in A$.
\item A realizer for such a thing is just $a$.
\item The 1-cells between $(A, \mathcal{A}, a)$ and $(B, \mathcal{B}, b)$ consist of a coded homotopy equivalence $(f,g,H,K): (A, \mathcal{A}) \to (B, \mathcal{B})$ and a 1-cell in $\mathcal{B}(fa,b)$.
\item The 2-cells in $E$ are as in $U$.
\end{itemize}
One can easily check that the projection $\pi: E \to U$ is a fibration of discrete sets, so it remains to verify that it is a univalent representation. So let $f: (B, \beta, \mathcal{B}) \to (A, \alpha, \mathcal{A})$ be a fibration of discrete sets with $\mathcal{B}(\pi,\pi') = \mathcal{A}(f\pi,f\pi')$, $B \subseteq A \times \mathbb{N}$, and $f$ the first and $\beta$ the second projection. Then for each $a \in A$ the fibre $B_a$ can be regarded as an object of $\mathbb{DISC}$ as between every two parallel 1-cells in $B$ there is a unique 2-cell. This means that the classifying map $A \to U$ can be defined on 0-cells by sending $a \in A$ to $B_a$. In addition, every $\pi: a \to a'$ induces an equivalence $\Gamma^B_\pi: B_a \to B_a'$ by \refprop{transportisanequivalence}, while every $n: \pi \Rightarrow \pi'$ induces a homotopy $\Gamma^B_\pi \simeq \Gamma^B_{\pi'}$ by \refprop{basicpropoftransport}.(2); this completes the definition of $A \to U$. The properties of transport imply that this map is a morphism in $\mathbb{EFF}_1$.

Finally, if $f, g: A \to U$ are two maps and $w: f^*\pi \to g^*\pi$ is an equivalence, then for each $a \in A$ the objects $f^{-1}(a)$ and $g^{-1}(a)$ are equivalent in $\mathbb{DISC}$ by $w_a$: using the definition of the path objects in $\mathbb{EFF}_1$ in \reflemm{pathobjects}, one can show that this induces a morphism $H: A \to PU$ which in turn induces an equivalence $f^{-1}(a) \to g^{-1}(a)$ which is fibrewise homotopic to $w$. In other words, $\pi$ is univalent.
\end{proof}

It is impossible to show that the class of discrete fibrations in $\mathbb{EFF}_1$ has a univalent representation in view of the following facts.

\begin{prop}{everythinggrpinEFF2}
In $\mathbb{EFF}_1$ every fibration is a fibration of groupoids.
\end{prop}
\begin{proof}
It is clear from the description of path objects in \reflemm{pathobjects} that $PA \to A \times A$ is always a fibration of (discrete) sets. Since the description of fibrewise path objects $P_I(A) \to A \times_I A$ for fibrations $A \to I$ is similar, we see that every fibration is a fibration of groupoids.
\end{proof}

\begin{coro}{nounivreprfordiscrfibrinEFF1}
The class of discrete fibrations in $\mathbb{EFF}_1$ does not have a univalent representation.
\end{coro}
\begin{proof}
If a class of small fibrations has a univalent representation $\pi: E \to U$ with $U$ a groupoid, then for every small object $A$, self-equivalence $w: A \simeq A$ and homotopies $H, K: 1 \simeq w$, we must have that $H \sim K$. For otherwise there would be a loop in the universe which is homotopic to the constant loop in two different ways, which is impossible if the universe is a groupoid. 

So consider the following discrete object $(A, \alpha, \mathcal{A})$ in $\mathbb{EFF}_1$: 
\begin{itemize}
 \item The set $A$ of 0-cells is $\{ 0, 1 \}$.
 \item The map $\alpha$ is the identity.
 \item The set $\mathcal{A}(i,j) = \{ 0, 1 \}$ for all 0-cells $i, j$.
 \item The set $\mathcal{A}(\pi,\pi')$ for parallel 1-cells $\pi,\pi'$ is $\{ 0 \}$ if $\pi = \pi'$ and empty otherwise.
 \item The identity 1- and 2-cells are 0.
 \item Composition of 1-cells is addition modulo 2.
\end{itemize}
This object has an endomorphism $w: A \to A$ which is obtained by sending 0-cells to themselves and 1-cells $\pi \in \mathcal{A}(i,j)$ to themselves if $i = j$ and to $1-\pi$ if $i \not= j$; because $w^2 = 1$, the map $w$ is an isomorphism and hence a self-equivalence.

There are two homotopies $H,K: 1 \simeq w$ obtained by putting $H_i = i: i \to i$ and $K_i = 1-i: i \to i$. These homotopies $H$ and $K$ are not homotopic to each other because $A$ only has identity 2-cells and the homotopies $H$ and $K$ are distinct.
\end{proof}

\begin{rema}{reprinEFF1}
 It is unclear to me if the class of discrete fibrations in $\mathbb{EFF}_1$ has a non-univalent representation.
\end{rema}

\section{Conclusion and directions for future research}

The contents of this paper can be summarised by the following two theorems:

\begin{theo}{summth0}
There exists a path category with homotopy $\Pi$-types $\mathbb{EFF}$ whose homotopy category is equivalent to the effective topos. Within this path category one can identify a class of small fibrations (\emph{viz.}, the ones which are discrete and propositional) which is impredicative and has a univalent representation. In addition, propositional resizing holds.
\end{theo}

\begin{theo}{summth1}
There exists a path category with homotopy $\Pi$-types $\mathbb{EFF}_1$ in which there is a class of small fibrations (\emph{viz.}, the fibrations of discrete sets) which is impredicative and has a univalent representation. In addition, propositional resizing holds.
\end{theo}

It is not entirely clear when two path categories should be considered equivalent, but one would expect that if one had such a notion, the path category of sets in $\mathbb{EFF}_1$ would be equivalent as a path category to $\mathbb{EFF} = \mathbb{EFF}_0$. Indeed, as I already suggested in the introduction, one would expect that there would be path categories $\mathbb{EFF}_n$ such that the $n$-types in $\mathbb{EFF}_{n+1}$ would be equivalent as a path category to $\mathbb{EFF}_n$. Ultimately, there should be a path category $\mathbb{EFF}_\infty$ such that each $\mathbb{EFF}_n$ is equivalent to the path category of $n$-types in $\mathbb{EFF}_\infty$. For this path category $\mathbb{EFF}_\infty$ one should have the following:

\begin{conj}{effinfty}
There exists a path category with homotopy $\Pi$-types $\mathbb{EFF}_\infty$ in which there is a class of small fibrations, the discrete ones, which is impredicative and has a univalent representation. These discrete fibrations can be defined using an internal orthogonality condition and for these discrete fibrations propositional resizing holds.
\end{conj}

In addition, the following problems remain.
\begin{description}
\item[Models of CoC and coherence issues] The categories $\mathbb{EFF}$ and $\mathbb{EFF}_1$ come close to being models of some version of CoC, but fall short in two ways. First of all, the version of CoC we have in mind is one in which all definitional equalities have been replaced by propositional ones. It remains to make this more precise and in particular write down a syntax for this version of CoC. Secondly, we should address coherence issues and formulate these categories as categories with families \cite{dybjer96,hofmann97b} or some other standard notion of model of type theory.
\item[Choice issues] We have used the axiom of choice in the proof of propositional resizing and in the proof of the fact that the discrete fibrations can be characterised by an internal orthogonality condition (see \refprop{discretealafreydinEFF} and \refprop{resizinginEFF} for the case of $\mathbb{EFF}$). It seems likely that both applications of the axiom of choice can be eliminated by changing the definition of $\mathbb{EFF}$ slightly. The idea would be to say that every object $(A, \alpha, \mathcal{A})$ in $\mathbb{EFF}$ should also come equipped with a section of the surjection $\alpha: A \to {\rm Im}(\alpha)$, but that morphisms are as before (so they do not have to preserve this additional structure). For the fibres of fibrations one should do something similar, so it might be that this would work best in the setting of categories with families. We leave the question whether this would give us a choice-free model of CoC to future work.
\end{description}

\bibliographystyle{plain} \bibliography{hSetoids}

\begin{thebibliography}{10}

\bibitem{vandenberg18}
B.~van~den Berg.
\newblock Path categories and propositional identity types.
\newblock {\em {ACM} Trans. Comput. Log.}, 19(2):15:1--15:32, 2018.

\bibitem{vandenberggarner11}
B.~van~den Berg and R.~Garner.
\newblock Types are weak {$\omega$}-groupoids.
\newblock {\em Proc. Lond. Math. Soc. (3)}, 102(2):370--394, 2011.

\bibitem{bergmoerdijk18}
B.~van~den Berg and I.~Moerdijk.
\newblock Exact completion of path categories and algebraic set theory. {P}art
  {I}: {E}xact completion of path categories.
\newblock {\em J. Pure Appl. Algebra}, 222(10):3137--3181, 2018.

\bibitem{bourke16}
J.~Bourke.
\newblock Note on the construction of globular weak omega-groupoids from types,
  topological spaces{$\ldots$}.
\newblock {\em Cah. Topol. G\'eom. Diff\'er. Cat\'eg.}, 57(4):281--294, 2016.

\bibitem{brown73}
K.S. Brown.
\newblock Abstract homotopy theory and generalized sheaf cohomology.
\newblock {\em Trans. Amer. Math. Soc.}, 186:419--458, 1973.

\bibitem{carbonietal88}
A.~Carboni, P.J. Freyd, and A.~Scedrov.
\newblock A categorical approach to realizability and polymorphic types.
\newblock In {\em Mathematical foundations of programming language semantics
  ({N}ew {O}rleans, {LA}, 1987)}, volume 298 of {\em Lecture Notes in Comput.
  Sci.}, pages 23--42. Springer, Berlin, 1988.

\bibitem{dybjer96}
P.~Dybjer.
\newblock Internal type theory.
\newblock In {\em Types for proofs and programs ({T}orino, 1995)}, volume 1158
  of {\em Lecture Notes in Comput. Sci.}, pages 120--134. Springer, Berlin,
  1996.

\bibitem{escardoxu15}
M.H. Escard{\'{o}} and C.~Xu.
\newblock The inconsistency of a {B}rouwerian continuity principle with the
  {C}urry-{H}oward interpretation.
\newblock In {\em 13th International Conference on Typed Lambda Calculi and
  Applications, {TLCA} 2015, July 1-3, 2015, Warsaw, Poland}, pages 153--164,
  2015.

\bibitem{fruminvandenberg17}
D.~Frumin and B.~van~den Berg.
\newblock A homotopy-theoretic model of function extensionality in the
  effective topos.
\newblock arXiv:1701.08369, 2017.

\bibitem{hofmann97b}
M.~Hofmann.
\newblock Syntax and semantics of dependent types.
\newblock In {\em Semantics and logics of computation ({C}ambridge, 1995)},
  volume~14 of {\em Publ. Newton Inst.}, pages 79--130. Cambridge Univ. Press,
  Cambridge, 1997.

\bibitem{hyland82}
J.M.E. Hyland.
\newblock The effective topos.
\newblock In {\em The L.E.J. Brouwer Centenary Symposium (Noordwij\-kerhout,
  1981)}, volume 110 of {\em Stud. Logic Foundations Math.}, pages 165--216.
  North-Holland Publishing Co., Amsterdam, 1982.

\bibitem{hyland88}
J.M.E. Hyland.
\newblock A small complete category.
\newblock {\em Ann. Pure Appl. Logic}, 40(2):135--165, 1988.

\bibitem{hylandetal90}
J.M.E. Hyland, E.P. Robinson, and G.~Rosolini.
\newblock The discrete objects in the effective topos.
\newblock {\em Proc. London Math. Soc. (3)}, 60(1):1--36, 1990.

\bibitem{joyal17}
A.~Joyal.
\newblock Notes on clans and tribes.
\newblock arXiv:1710.10238, 2017.

\bibitem{kraussattler15}
N.~Kraus and C.~Sattler.
\newblock Higher homotopies in a hierarchy of univalent universes.
\newblock {\em ACM Trans. Comput. Log.}, 16(2):Art. 18, 12, 2015.

\bibitem{vanoosten08}
J.~van Oosten.
\newblock {\em Realizability: an introduction to its categorical side}, volume
  152 of {\em Studies in Logic and the Foundations of Mathematics}.
\newblock Elsevier B. V., Amsterdam, 2008.

\bibitem{vanoosten15}
J.~van Oosten.
\newblock A notion of homotopy for the effective topos.
\newblock {\em Math. Structures Comput. Sci.}, 25(5):1132--1146, 2015.

\bibitem{ortonpitts18}
I.~Orton and A.M. Pitts.
\newblock Decomposing the univalence axiom.
\newblock In A.~Abel, F.~{Nordvall Forsberg}, and A.~Kaposi, editors, {\em 23rd
  International Conference on Types for Proofs and Programs (TYPES 2017),
  Post-Proceedings Volume}, volume~? of {\em Leibniz International Proceedings
  in Informatics (LIPIcs)}, pages 7:1--7:19, Dagstuhl, Germany, 2018. Schloss
  Dagstuhl--Leibniz-Zentrum f\"ur Informatik.

\bibitem{univalent13}
The Univalent~Foundations Program.
\newblock {\em Homotopy type theory---univalent foundations of mathematics}.
\newblock The Univalent Foundations Program, Princeton, NJ; Institute for
  Advanced Study (IAS), Princeton, NJ, 2013.

\bibitem{robinsonrosolini90}
E.P. Robinson and G.~Rosolini.
\newblock Colimit completions and the effective topos.
\newblock {\em J. Symbolic Logic}, 55(2):678--699, 1990.

\bibitem{rosolini16}
G.~Rosolini.
\newblock The category of equilogical spaces and the effective topos as
  homotopical quotients.
\newblock {\em J. Homotopy Relat. Struct.}, 11(4):943--956, 2016.

\bibitem{troelstravandalen88}
A.~S. Troelstra and D.~van Dalen.
\newblock {\em Constructivism in mathematics. {V}ol. {II}}, volume 123 of {\em
  Stud. Logic Foundations Math.}
\newblock North-Holland Publishing Co., Amsterdam, 1988.

\bibitem{uemura18}
T.~Uemura.
\newblock Cubical assemblies and independence of the propositional resizing
  axiom.
\newblock arXiv:1803.06649, 2018.

\end{thebibliography}

\end{document}